\documentclass[11pt]{amsart}

\usepackage{amsfonts}
\usepackage{amscd}
\usepackage{amssymb}
\allowdisplaybreaks[3]

\usepackage{color}

\newtheorem{thm}{Theorem}[section]
\newtheorem{cor}[thm]{Corollary}
\newtheorem{lem}[thm]{Lemma}

\newtheorem{prop}[thm]{Proposition}
\theoremstyle{definition}

\theoremstyle{remark}
\newtheorem{rem}[thm]{Remark}
\numberwithin{equation}{section}

\newcommand{\R}{\mathbb R}
\newcommand{\Z}{\mathbb Z}
\newcommand{\Na}{\mathbb N}

\newcommand{\si}{\sigma}

\newcommand{\la}{\lambda}
\newcommand{\C}{{\mathbb C}}

\newcommand{\pa}{\partial }
\newcommand{\La}{\langle}
\newcommand{\Ra}{\rangle}

\newcommand{\Hc}{\mathcal H}

\newcommand{\sd}{{\mathbf S}^{d-1}}

\newcommand{\om}{ \omega}
\newcommand{\D}{\Delta}
\newcommand{\Dk}{\Delta_{\kappa}}
\newcommand{\dl}{\delta }
\newcommand{\ap}{\alpha}
\newcommand{\bt}{\beta }

\newcommand{\K}{\kappa }
\newcommand{\gm}{\gamma}
\newcommand{\Gm}{\Gamma}
\newcommand{\ve}{\varepsilon}

\renewcommand{\Re}{\operatorname{Re}}

\title[Ces\`{a}ro means for Dunkl--Hermite expansions]
{Mixed norm estimates for the Ces\`{a}ro means associated with  Dunkl--Hermite expansions}

\author[Pradeep B., L. Roncal, and S. Thangavelu]{Pradeep Boggarapu, Luz Roncal, and Sundaram Thangavelu}

\address[Pradeep B. and S. Thangavelu]{Department of Mathematics\\
 Indian Institute of Science\\
560 012 Bangalore, India}

\email{\{pradeep,veluma\}@math.iisc.ernet.in}

\address[L. Roncal]{Departamento de Matem\'aticas y Computaci\'on\\
    Universidad de La Rioja\\
    26004 Logro\~no, Spain}
\email{luz.roncal@unirioja.es}

\thanks{All the three authors were supported by the J. C. Bose
Fellowship of the third author from the Department of Science and
Technology, Government of India. The second author was also
supported by grant MTM2012-36732-C03-02 from Spanish Government}

\begin{document}

\maketitle

\begin{abstract}
Our main goal in this article is to study mixed norm estimates for  the Ces\`{a}ro means associated with Dunkl--Hermite expansions on $\R^d$. These expansions arise when one considers the Dunkl--Hermite operator (or Dunkl harmonic oscillator) $H_{\K}:=-\D_{\K}+|x|^2$, where $\D_{\K}$ stands for the Dunkl--Laplacian. It is shown that the desired mixed norm estimates are equivalent to vector-valued inequalities for a sequence of Ces\`{a}ro means for Laguerre expansions with shifted parameter. In order to obtain such vector-valued inequalities, we develop an argument to extend these Laguerre operators for complex values of the parameters involved and apply a version of three lines lemma.

\end{abstract}
\section{Introduction and main results}

The Dunkl operators were introduced by C. F. Dunkl in \cite{DU}, where
he built a framework for a theory of special functions and integral transforms in several variables related to reflection groups. Such operators are relevant in physics, namely for the analysis of quantum many body systems of Calogero--Moser--Sutherland type
(see \cite{DV,LV}). From the mathematical analysis point of view, the importance of Dunkl operators lies on the fact that
they generalize the theory of symmetric spaces of Euclidean type. There is a vast literature related to Dunkl transform and Dunkl Laplacian, see for instance \cite{AS,DW,DJ,HS,RO1,TX,TX1}.

In \cite{RO} M. R\"{o}sler studied the Dunkl--harmonic oscillator (which we will also call Dunkl--Hermite operator)
\begin{equation*}
H_{\K}:=-\D_{\K}+|x|^2,
\end{equation*}
 where $\D_{\K}$ stands for the Dunkl--Laplacian \eqref{eq:DunklLap}, and introduced the \textit{Dunkl--Hermite functions} $\Phi_{\mu,\K}$, $\mu\in\Na^d$, as eigenfunctions of $H_{\K}$. When $\K$, which is called a multiplicity function (see \eqref{eq:k}), is the null function, the situation is reduced to the standard Hermite operator
$H$ and $\Phi_{\mu,\K}$ become the usual Hermite functions $\Phi_{\mu}$, see \cite[page 521]{RO}. The set
$\{\Phi_{\mu,\K}\}_{\mu \in \Na^d}$ forms an orthonormal basis for $L^2(\R^d, h_{\K}^2\, dx)$ where $h_{\K}^2$ is a suitable
weight function defined in terms of the corresponding reflection group and the multiplicity function $\K$, see precise definitions in Section \ref{Prelim}.  Thus we have the orthogonal expansion
$$ f=\sum_{\mu \in \Na^d}(f, \Phi_{\mu,\K})\Phi_{\mu,\K} $$
which converges to $f$ in $L^2(\R^d, h_{\K}^2\, dx)$. Here $(\cdot,\cdot)$ is the inner product in $L^2(\R^d, h_{\K}^2\, dx)$. In short we can rewrite this as

 \begin{equation}
 \label{eq:L2expn}
 f=\sum_{j=0}^{\infty}P_{j, \K}f, \,\,\,\,\text{where}  \,\,\,\,P_{j, \K}f=\sum_{|\mu|=j}(f,\Phi_{\mu,\K})\Phi_{\mu,\K}.
 \end{equation}
 It is known that the Dunkl--Hermite functions are of Schwartz class and hence the projections $P_{j, \K}f$ in \eqref{eq:L2expn} make sense for any  $f \in L^p(\R^d, h_{\K}^2\, dx) $. However, when $p\neq 2$, we do not have any convergence results for $\sum_{j=0}^{\infty}P_{j, \K}f$. In  the case of $\K\equiv0$, where the above series reduces to standard Hermite expansions, it is also well known that the Hermite expansions fail to converge in $L^p(\R^d, dx)$ for $p\neq 2 $, unless $d=1$. Even when $d=1$, the series converges in $L^p(\R, dx)$ (or, equivalently, the corresponding partial sum operators are uniformly  bounded) if and only if $\frac43<p<4$ according to a  theorem by R. Askey and S. Wainger \cite{AW}.

In the absence of convergence results for the partial sums, we are led to consider other summability methods such as Ces\`{a}ro means or Bochner--Riesz means. For $N\in \Na$ and $\dl>0$ we define the Ces\`{a}ro means of order $\dl$ associated with the Dunkl--Hermite expansions by
 \begin{equation} \label{eq:DunklCesaro}
 \si_{N,\K}^{\dl}f(x):= \frac{1}{A_N^{\dl}}\sum_{j=0}^{N}
 A_{N-j}^{\dl}P_{j,\K}f(x)
\end{equation}
 where
$$  A_{j}^{\dl}:=\binom{j+\dl}{j}=\frac{\Gm(j+\dl+1)}{\Gm(j+1)\Gm(\dl+1)}  $$
are the binomial coefficients. When $\K\equiv0$ and $d\geq 2$, the operators $\si_{N,0}^{\dl}f$ (that we will simply denote by $\si_{N}^{\dl}f$) are nothing but the Ces\`{a}ro means for the standard Hermite expansions. They converge to $f$ in $L^p(\R^d, dx)$,  $1\leq p <\infty$, whenever $\dl >\frac{d-1}{2}$, see \cite{ST,ST1}. Actually, more precise results are known, giving critical indices of summability for any given $p$, $1\leq p <\infty$.

In the one dimensional case there is only one reflection group, viz.\ $ \Z_2 $, and the generalized Hermite expansion in this case
has been studied by \'O. Ciaurri and J. L. Varona \cite{CV}. Therein, the authors studied weighted norm inequalities for the
Ces\`{a}ro means.  However, for  the higher dimensional case we are not aware of any work dealing with Ces\`{a}ro or Riesz
means associated with Dunkl--Hermite expansions, which is the main concern of the present work.

The techniques used to study Ces\`{a}ro means $\si_{N}^{\dl}f$ for the standard Hermite expansions are not available in the case of Dunkl--Hermite expansions. This is mainly due to the lack of explicit formulas for the Dunkl--Hermite functions $\Phi_{\mu,\K}$ and their asymptotic properties.

However, the situation changes if we express the Ces\`{a}ro means in a convenient way, by separating variables in polar coordinates. Indeed, the basic fundamental idea is to write down the Ces\`{a}ro means $\si_{N,\K}^{\dl}f$ in terms of \textit{spherical $h$-harmonics}. The spherical $h$-harmonics (or simply \textit{$h$-harmonics}) are the restrictions of solid $h$-harmonics to $\sd$, where $\sd$ is the unit sphere in $\R^d$, $d\ge2$  (a good reference for $h$-harmonics is \cite[Chapter 5]{DUX}). By solid $h$-harmonics we mean homogeneous polynomials $P(x) $  satisfying $\Dk P(x)=0 $.  Let $\Hc_{m} $ be the space of all $h$-harmonics of degree $m $. Then the space $L^2(\sd, h^2_{\K}\,d\si ) $ is the orthogonal direct sum of the finite dimensional spaces $\Hc_{m}$ over $m=0,1,2,\ldots$. Thus there is an orthonormal basis $\{Y_{m,j}:  j= 1,2,\ldots,\,d(m), m = 0,1,2,\ldots  \}$, where
\begin{equation}\label{eq:dim}
d(m)=\operatorname{dim} (\Hc_{m}),
\end{equation} for $L^2(\sd, h^2_{\K}d\,\si )$
so that for each $m$,  $\{Y_{m,j}: j=1,2,\ldots,d(m)\}$ is an orthonormal basis of $h$-harmonics of degree $j$ for $\Hc_{m}$. For $x\in \R^d$,  if we take $x=rx'$, $r\in (0,\infty)$, $x'\in \sd$, then the $h$-harmonic expansion of a function $f$ is given by
\begin{equation*}
f(rx')=\sum_{m=0}^\infty\sum_{j=1}^{d(m)}f_{m,j}(r)Y_{m,j}(x'),
\end{equation*}
where the $h$-harmonic coefficients are
\begin{equation}\label{eq:h-harcoeff}
f_{m,j}(r)=\int_{\sd}f(rx')Y_{m,j}(x')h_{\K}^2(x')d\si(x').
\end{equation}

Therefore, due to the ortogonality of the $h$-harmonics, it is natural to introduce the mixed norm spaces for Dunkl--Hermite expansions $L^{p,2}(\R^d, h_{\K}^2\, dx)$, namely, the space of all functions on $\R^d$ for which
\begin{equation}\label{mixedNorm}
\|f\|_{(p,2)}:=\bigg(\int_0^\infty\bigg (\int_{\sd}|f(rx')|^2h_{\K}^2(x')\,d\si(x') \bigg )^{\frac{p}{2}}r^{d+2\gm-1}\,dr \bigg)^{\frac{1}{p}}
\end{equation}
are finite. Here $\gm$ is a positive number intrinsic to the underlying multiplicity function $\K$, see Section \ref{Prelim} for the definition. Our main result is the following.

\begin{thm}\label{main}
Let $d\geq 2$ and $ \dl>\frac{d+2\gm-1}{2}$ where $\gm$ is as in \eqref{gamma}. Then, for any $1<p<\infty $, we have the uniform estimates
$$\|\si_{N,\K}^{\dl}f\|_{(p,2)}\leq C\|f\|_{(p,2)} .$$
Consequently, $\si_{N,\K}^{\dl}f$ converges to $f$ in $L^{p,2}(\R^d, h_{\K}^2 dx)$ as $N\rightarrow \infty$.
\end{thm}

We will sketch the outline of the proof. The first step is to write down the Ces\`{a}ro means $\si_{N,\K}^{\dl}f$ evaluated at $x=rx'$ in terms of spherical $h$-harmonics. Then, a Funk--Hecke formula is used to rewrite the projection of the kernel associated with $\si_{N,\K}^{\dl}f$ in such a basis as a sum of shifted Ces\`{a}ro kernels for the usual Hermite expansions. Indeed, we will prove that
$$\si_{N,\K}^{\dl}f(rx') = \sum_{m=0}^N \sum_{j=1}^{d(m)}T_{N,m}^{\dl,\gm}
f_{m,j}(r) Y_{m,j}(x').$$
In the expression above, $T_{N,m}^{\dl,\gm}$ are the linear operators defined by
\begin{equation}\label{eq:alterCesaro}
T_{N,m}^{\dl,\gm} f_{m,j}(r):=\int_0^\infty K_{N,m}^{\dl,\gm}(r, s)
f_{m,j}(s)s^{d+2\gm-1}ds,
\end{equation}
where  $f_{m,j}(r)$ are the $h$-harmonic coefficients of $ f$ as in \eqref{eq:h-harcoeff}. The kernels $ K_{N,m}^{\dl,\gm} $ can be written in terms of the kernels of  either the $d$-dimensional or $d+1$-dimensional Ces\`{a}ro means for the standard Hermite expansions with shifted parameters. See Corollary \ref{kernelexp}. But now, as the standard Hermite polynomials evaluated at $x$ are connected with the Laguerre polynomials evaluated at $ |x|^2$, the kernels $ K_{N,m}^{\dl,\gm} $ can also be written in terms of Laguerre functions. See Proposition \ref{prop:relKernels}. So when trying to prove the $(p,2)$ boundedness of the Ces\`{a}ro operator, we end up with a vector-valued extension for a sequence of operators associated to Laguerre expansions.

\begin{thm}[Vector-valued inequalities for $T_{N,m}^{\delta,\gm}$]
\label{Th:Vvine}
Let $0\leq \gm \leq \frac12 $, $\dl>\tfrac{d+2\gm-1}{2}$ and $T_{N,m}^{\delta,\gm}$
be the operators defined in \eqref{eq:alterCesaro}. Then, for any $1<p<\infty$,
there is a constant $C$ independent of $N$ such that
$$\Big\|\Big(\sum_{m=0}^\infty\sum_{j=1}^{d(m)}|T_{N,m}^{\dl,\gm}f_{m,j}(r)|^2\Big)^{1/2}
\Big\|_{L^p(\R^+, r^{d+2\gamma-1}dr)}\le
C\Big\|\Big(\sum_{m=0}^\infty\sum_{j=1}^{d(m)}|f_{m,j}(r)|^2\Big)^{1/2}\Big\|_{L^p(\R^+,
r^{d+2\gamma-1}dr)},$$
where $d(m)$ is as in \eqref{eq:dim}, for any sequence of functions $f_{m,j}\in L^p(\R^+,r^{d+2\gamma-1}dr)$ for which the right hand side is finite.
\end{thm}

Then, it will be seen that Theorem \ref{main} is a direct consequence of Theorem \ref{Th:Vvine}.

In order to prove Theorem \ref{Th:Vvine}, we need to use a sophisticated version of the three lines lemma. This will recquire to extend the Ces\`{a}ro means for the usual Hermite expansions for complex values of $\delta$ and also complexify the type of the Laguerre functions involved.

The natural question that arises immediately concerns the convergence of Ces\`{a}ro means in $L^p(\R^d,h_{\kappa}^2\,dx)$, when $\kappa\neq 0$. Nothing is known about this, neither the techniques that could be used to solve the problem.

The paper is organized as follows. In Section \ref{Prelim} we recall basics about the general Dunkl context and the
Funk--Hecke identity. Section \ref{sec:standardHermite} is devoted to the study of Ces\`{a}ro means for the standard Hermite expansions, and a vector-valued inequality for an operator related to these Ces\`{a}ro means is proved. In Section \ref{sec:cesaroDunkl} we express the Ces\`{a}ro kernels for the Dunkl--Hermite expansions in terms of the same for the standard Hermite expansions, and we bring out the connection with Laguerre expansions, which allows us to introduce an analytic family of operators. With this, we prove the main results in Section \ref{subsec:Mainproofs}. For the sake of reading, we move the proofs of certain propositions in Section \ref{sec:cesaroDunkl} to Section \ref{sec:propositions}. Section \ref{sec:technical} contains also some technical results and integral formulas for Bessel functions which involve ultraspherical polynomials with real and complex parameters.

\subsubsection*{Notation.} Throughout the paper, we will use the following notation. In general, all the operators, kernels and functions are defined in the ambient space $\R^d$, so in these cases we will mostly omit the dimensional parameter $d$. For instance, we will write just $\Phi_{k}$ or $\si_{N}^{\dl}$. Nevertheless, whenever the dependence of the kernels on the dimension is explicitly needed we will add a superindex $d$, namely we will write $\Phi_{k}^{(d)}$ or $\si_{N}^{\dl, d}$. Moreover, when referring to these objects in $\R^{d+1}$, we will always explicitly add a superindex $d+1$ (like $\Phi_{k}^{(d+1)}$ or $\si_{N}^{\dl, d+1}$), to denote the Hermite functions, Ces\`{a}ro means and kernels in $\R^{d+1}$. The number  of operators and parameters appearing in this work is quite large, and we tried to reduce the notation to a minimum. Because of this, we present in Section \ref{subsec:Mainproofs} a table with a summary of the most important operators and kernels involved directly in the proof of the main theorems. For $1\le p\le\infty$, $p'$ will denote its conjugate, $1/p+1/p'=1$. Moreover, we shall write $C$ to denote positive constants independent of significant quantities the meaning of which can change from one occurrence to another.

\section{The general Dunkl setting}
\label{Prelim}

For completeness, in this section we collect several facts concerning the general Dunkl
setting and the Dunkl harmonic oscillator. For a more detailed exposition on these topics,
we refer the reader to \cite{DU, DUX, RO}.

We use the notation $\La \cdot , \cdot \Ra $ for the standard
 inner product on $\R^d $.  For $\nu\in \R^d\setminus\{0\}$, we denote by $\si_{\nu}$ the orthogonal reflection in the hyperplane perpendicular to $\nu$, i.e.,
$$\si_{\nu}(x)= x - 2 \; \frac{\La \nu, x \Ra}{|\nu |^2} \nu .$$
A finite subset $R\subset  \R^d\setminus\{0\}$ is a root system if $\si_{\nu}(R)=R$, for all $\nu\in R$. Each root system can be written as a disjoint union $R=R_+\cup (-R_+)$, where $R_+$ and $-R_+$ are separated by a hyperplane through the origin. Such $R_+$ is called the set of all positive roots in $R$. The group $G$ generated by the reflections $\{\si_{\nu}:\nu\in R \}$ is called the reflection group or Coxeter group  associated with $R$. A function
\begin{equation} \label{eq:k}
\K:R\to [0,\infty)
\end{equation}
which is invariant under the action of $G$ on the root system $R$ is called a multiplicity function.  Let $T_j$, $j=1,2, \ldots d$, be the difference-differential operators  defined by
$$ T_jf(x) = \frac{\pa f}{\pa x_j}(x) +\sum_{\nu \in R_+}\K(\nu)\nu _j \frac{f(x)-f(\si_{\nu}x)}{\La \nu, x \Ra}. $$
These operators, known as Dunkl operators, form a family of commuting operators. The Dunkl Laplacian $\Dk $ is then defined to be the operator
\begin{equation}\label{eq:DunklLap}
\Dk = \sum _{j=1}^d T_j^2
\end{equation}
which can be explicitly calculated, see \cite[Theorem 4.4.9]{DUX}. It is known that the operators $T_j$ have a joint eigenfunction
$E_{\K}(x,y)$ satisfying
$$ T_j E_{\K}(x,y)=y_j E_{\K}(x,y), \qquad  j=1,\ldots,d. $$
The function $(x,y)\mapsto E_{\K}(x,y)$ is called the \textit{Dunkl kernel} or the \textit{generalized exponential kernel} on $\R^d\times \R^d$, which is the generalization of the exponential function $e^{\langle x,y\rangle}$. Associated with the root system $R$ and the multiplicity function $\K $, the weight function $h^2_{\K}(x)$ is defined by
$$h^2_{\K}(x):=\prod_{\nu \in R_+}|\La x, \nu \Ra |^{2\K(\nu)}. $$
The nonnegative real number
\begin{equation}\label{gamma}
\gm = \sum_{\nu \in R_+}\K(\nu)
\end{equation}
defined in terms of the multiplicity function $\K(\nu)$ plays an important role in Dunkl theory. Note that $h^2_{\K}(x)$ is homogeneous of degree $2\gm $, which motivates the definition of mixed norm spaces as in \eqref{mixedNorm}.

In the Dunkl setting we have a \textit{Funk--Hecke formula for $h$-harmonics}.
To state such formula, we need to recall the intertwining operator in the Dunkl setting. It is known that there is an operator $V_{\K}$ satisfying $T_jV_{\K}=V_{\K} \frac{\pa}{\pa x_j} $. However, the explicit form of $V_{\K}$ is not known, except in a couple of simple cases, but it is a useful operator. In particular, the Dunkl kernel is given by $E_{\K}(x,y) = V_{\K}e^{\La \cdot ,y \Ra}(x) $. The Funk--Hecke formula for $h$-harmonics is as follows (see \cite[Theorem 7.2.7]{DX} or \cite[Theorem 5.3.4]{DUX}).

\begin{thm}[Funk--Hecke for $h$-harmonics]\label{Th:FunkHeckeh}
Let $f$ be a continuous function defined on $[-1, 1]$ and
$\la=\frac{d}{2}+\gm-1$. Then for every $Y_{m,j} \in \Hc_{m}$,
$$\int_{\sd}V_{\K}f(\La x',\cdot \Ra)(y')Y_{m,j}(y')h_{\K}^2(y')d\si(y')=\Lambda_m^{\K}(f)Y_{m,j}(x') $$
where $\Lambda_m^{\K}(f) $ is a constant defined by
$$\Lambda_m^{\K}(f)= \frac{\om_d^{\K}\Gm(\la+1)}{\sqrt{\pi}\Gm(\la+1/2)} \int_{-1}^1f(u)P_m^{\la}(u)(1-u^2)^{\la-\frac{1}{2}}\,du $$
with
\begin{equation}\label{constSphere}
\om_d^{\K}:=\int_{\sd}h_{\K}^2(\om)d\si(\om).
\end{equation}
\end{thm}
By applying Theorem \ref{Th:FunkHeckeh} to  the function $f(t) =e^{rst},~~~r, s \geq 0,$ and using  the fact $V_{\K}f(\La x', y'\Ra) = E_{\K}(rx', sy')$, we immediately obtain the following.

\begin{cor}[Funk--Hecke for the Dunkl kernel] \label{eq:Hecke}
 Let $\la=\frac{d}{2}+\gm-1$. Then for every $Y_{m,j} \in \Hc_{m}$,
\begin{align*}
\int_{\sd}E_{\K}(rx', sy')&Y_{m,j}(y')h_{\K}^2(y')\,d\si(y')\\ \notag &=\frac{\om_d^{\K}\Gm\big (\tfrac{d}{2}+\gm\big )}{\sqrt{\pi}\Gm\big (\tfrac{d-1}{2}+\gm\big )} \left(\int_{-1}^{1}e^{rsu}P_m^{\la}(u)(1-u^2)^{\la-\frac12}\,du\right)~ Y_{m,j}(x'),
\end{align*}
where $\om_d^{\K}$ is as in \eqref{constSphere}.
\end{cor}

\section{Ces\`{a}ro means for the standard Hermite and extension to complex parameters}
\label{sec:standardHermite}

One of the key points to prove Theorem \ref{main} is to express Ces\`{a}ro kernels for the Dunkl--Hermite expansions in terms of
Ces\`{a}ro kernels for the standard Hermite expansions and then extend these operators for complex values of the parameters
involved. In this section we recall some basic results concerning the $L^p$ boundedness of Ces\`{a}ro means for the standard Hermite expansions and prove others concerning the extended operators.

\subsection{Ces\`{a}ro and Bochner--Riesz means for the standard Hermite expansions}
\label{subsec:standardHermite-BR}

As explained in the introduction, when $\K\equiv0$ the Dunkl--Hermite functions reduce to the standard Hermite functions $\Phi_{\mu}$ on $\R^d$. Ces\`{a}ro means of order $\dl\geq 0$ associated with the Hermite expansions (or just standard Ces\`{a}ro means) are then defined by
$$ \si_{N}^{\dl}f=\frac{1}{A_N^{\dl}}\sum_{j=0}^N A_{N-j}^{\dl}P_jf, $$
$P_jf$  being the corresponding projections. The operators $\si_{N}^{\dl}$ can be described as integrals operators with a kernel
$\sigma_N^{\dl}(x,y)$, which is explicitly given by
$$\si_N^{\dl}(x,y)=\frac{1}{A_N^{\dl}}\sum_{j=0}^N A_{N-j}^{\dl}\Phi_{j}(x,y),$$
where $\Phi_j(x,y)$ is the kernel of the $j$th projection associated with the Hermite
operator (see \cite[page 6]{ST}). For $|w|<1$, Mehler's formula for $\Phi_j(x,y)$ reads as
\begin{equation}
\label{eq:Mehler}\sum_{j=0}^\infty \Phi_{j}(x,y)w^j =
\pi^{-\frac{d}{2}} (1-w^2)^{-\frac{d}{2}}e^{-\frac{1}{2}
\big (\frac{1+w^2}{1-w^2} \big )(|x|^2+|y|^2)+\frac{2w}{1-w^2}x\cdot y}.
\end{equation}
From Mehler's formula it follows that $\Phi_{j}(rx',sy')$, and consequently $\sigma_N^{\dl}(rx',sy')$, is a function of $r,s $
and $ u:=x'\cdot y'$. Hence, sometimes we will write $\Phi_{j}(r,s;u)$ instead of $\Phi_{j}(rx',sy')$ and $ \sigma_{N}^{\delta}(r,s;u) $ instead of $ \sigma_{N}^{\dl}(rx',sy')$.

We also introduce the Bochner--Riesz means associated with the Hermite expansions as
$$ S_R^{\dl}f = \sum_{j=0}^{\infty} \bigg (1-\frac{2j+d}{R}\bigg )^{\dl}_{+}P_jf, $$
where $R>0$ and $(1-s)_+=\max\{1-s,0\}$. In the literature, the boundedness of both Ces\`{a}ro and Bochner--Riesz means have been studied. Their behaviour are similar in the sense that is possible to express the Ces\`{a}ro means $\si_{N}^{\dl}f$ in terms of $S_R^{\dl}f$ and vice-versa. Indeed, we have the following theorem due to J. J. Gergen \cite{JJG}, adapted to our context.
\begin{thm}[Gergen]
Let $ m $ be the integral part of $ \delta $. Then there exist two functions $U$ and $V$, $U(x)=O(x^{-2})$ as $x
\rightarrow \infty$, $U(x)=O(x^{m-\dl+1}) $ as $x \rightarrow 0 $, and $V(x)=O(x^{-2})$ as $x \rightarrow \infty $, $V(x)=O(x^{\dl}) $ as $x \rightarrow 0$, such that
$$S_{R}^{\dl}(x,y)=R^{-\dl}\sum_{k\leq R}V(R-k)A_k^{\dl}\si_{k}^{\dl}(x,y)$$ and
$$\si_N^{\dl}(x,y)=\frac{1}{A_N^{\dl}}\int_0^{N+1}U(N+1-t)t^{\dl}S_{t}^{\dl}(x,y)\,dt,\quad \hbox{for } N=0,1,\ldots$$
\end{thm}
In view of Gergen's theorem, we can readily prove a version of \cite[Theorem 3.3.3]{ST} (that states pointwise estimates for Bochner--Riesz) for Ces\`{a}ro means.
\begin{thm}\label{pointwiseCesaro}
Let $d\geq 2$ and $\dl >\frac{d-1}{2}$. Then for any $q\geq 2$ and $f\in L^q(\R^d)$ we have the pointwise inequality
$$\sup_{N}|\si_N^{\dl}f(x)|\leq C M_qf(x) $$
where $M_qf(x)=\big(M|f|^q\big)^{\frac{1}{q}}(x),$ $M$ being the Hardy--Littlewood maximal function.
\end{thm}

\subsection{Extension of standard Ces\`{a}ro means to complex parameters}
\label{subsec:standardHermite-BR}

 In order to prove Theorem \ref{main} we need to consider Ces\`{a}ro means $\si_{N}^{\dl}$ when $\dl$ is complex. Let us
define
\begin{equation}\label{zetaC}
\dl(\zeta):=\frac{d-1}{2}+\zeta \quad \quad \hbox{for }~\zeta \in \C
\end{equation}
and consider $\si_{N}^{\dl(\zeta)}f $ with $\Re(\dl(\zeta))\geq 0 $. Let us also recall the definition of a function of admissible growth: we say that a function $F(y)$, $y\in \R$, is of admissible growth if there exist constants $a<\pi$ and $b>0$ such that $|F(y)|\le e^{b e^{a|y|}}$. As a consequence of Theorem \ref{pointwiseCesaro} we get the following result.
\begin{thm}\label{maximal}
Let $d\geq 2$. Then for any $q\geq 2$ and $f\in L^q(\R^d)$ we have the pointwise inequality
\begin{equation}\label{ineq:1}
\sup_{N}\big |\si_{N}^{\dl(i\bt+\varepsilon)}f(x) \big |\leq C_{\varepsilon}(\bt)M_qf(x)
\end{equation}
for a fixed $\varepsilon >0$ and $\bt \in \R$. Here, the function $C_{\varepsilon}(\bt)$ is of admissible growth. Moreover, the operator $\displaystyle \sup_{N}\big |\si_{N}^{\dl(i\bt+\varepsilon)}f \big | $ is bounded on $L^p(\R^d)$ for any $p>2$.
\end{thm}
\begin{proof}
 From the definition of $\si_{N}^{\dl}$ it follows that
 \begin{equation}
 \label{complexReal}
 \si_{N}^{\dl(i\bt+\varepsilon)}f(x)=\frac{1}{A_N^{\dl(i\bt+\varepsilon)}}
 \sum_{j=0}^N A_{N-j}^{i\bt+\frac{\varepsilon}{2}-1}A_j^{\dl(\varepsilon/2)}
 \si_j^{\dl(\varepsilon/2)}f(x).
 \end{equation}
 Indeed, the right hand side in the identity above is
$$ \frac{1}{A_N^{\dl(i\bt+\varepsilon)}} \sum_{j=0}^N A_{N-j}^{i\bt+\frac{\varepsilon}{2}-1}
\sum_{k=0}^jA_{j-k}^{\delta(\varepsilon/2)}P_kf(x)=\frac{1}{A_N^{\dl(i\bt+\varepsilon)}}
\sum_{k=0}^NP_kf(x) \sum_{j=k}^N A_{N-j}^{i\bt+\frac{\varepsilon}{2}-1}A_{j-k}^{\delta(\varepsilon/2)},$$
so \eqref{complexReal} will be proved once we check that
$$A_{N-k}^{\dl(i\bt+\varepsilon)}=\sum_{j=k}^N A_{N-j}^{i\bt+\frac{\varepsilon}{2}-1}A_{j-k}^{\delta(\varepsilon/2)},$$
or equivalently
\begin{equation}\label{generating}
A_{N}^{\dl(i\bt+\varepsilon)}=\sum_{j=0}^N A_{N-j}^{i\bt+\frac{\varepsilon}{2}-1}A_{j}^{\delta(\varepsilon/2)}.
\end{equation}
Recall the following basic facts concerning generating functions. Since
$(1-w)^{-r-1}=\sum_{k\ge0}\binom{r+k}{k}w^k$ and $(1-w)^{-s-1}=\sum_{k\ge0}\binom{s+k}{k}w^k$, if we multiply these together, we get $(1-w)^{-(r+s+1)-1}=(1-w)^{-r-1}(1-w)^{-s-1}$. Equating coefficients gives us
\begin{equation}
\label{binomIdent}
\binom{r+s+1+n}{n}=\sum_{k=0}^n\binom{r+n-k}{n-k}\binom{s+k}{k}.
\end{equation}
With this, we see that \eqref{generating} is true.

Also, by Lemma \ref{lem:bin.est}, we have that
 $$ \sum_{j=0}^N\big |A_{N-j}^{i\bt+\frac{\varepsilon}{2}-1} \big | A_j^{\dl(\varepsilon/2)}
 \leq C_{\varepsilon}(\bt)\big |A_N^{\dl(i\bt+\varepsilon)} \big |. $$

Now, since $\dl(\varepsilon/2)>\frac{d-1}{2}$, it follows from Theorem \ref{pointwiseCesaro} that $\displaystyle \sup_{j}\big|\si_{j}^{\dl(\varepsilon/2)}f(x)\big|\leq C M_qf(x)$, so \eqref{ineq:1} is proved.

Finally, the last statement of Theorem \ref{maximal} follows from the fact that the maximal function $M_q$ is bounded on $L^p(\R^d)$ for any $p>q\geq 2 $.
\end{proof}

\subsection{A vector-valued inequality for an operator connected to standard Ces\`aro means}
\label{subsec:standardHermite-BR}

Further, we introduce one more operator $\mathbb{S}_N^{\dl} $ related to $\si_N^{\dl}$ as follows. For any $f\in L^p(\R^+, dr) $ we define
\begin{equation}\label{eq:supremeOp}
\mathbb{S}_N^{\dl} f(r):=r^{\frac{d-1}{p}}\int_{0}^{\infty}s^{\frac{d-1}{p'}}
\Big (\int_{-1}^1 \big|\sigma_N^{\dl}(r,s;u) \big|(1-u^2)^{\frac{d-3}{2}}\,du\Big )f(s)ds.
\end{equation}
We will require a vector-valued inequality (Theorem \ref{Th:MVvine}) for the maximal function associated with $\mathbb{S}_N^{\dl} $. In order to get this, recall the following result about vector-valued extensions for general bounded operators by J. L. Krivine (see \cite{Kr} or \cite[Thm.\ 1.f.14]{LinTza}).

\begin{thm}[Krivine]\label{Kr}
Let $X$ and $Y$ be two Banach lattices and let $T:X\mapsto Y$ be a bounded linear operator. Then, for every choice of $\{x_i\}_{i=1}^n$ in $X$, we have
$$\bigg\|\Big(\sum_{i=1}^n|Tx_i|^2\Big)^{1/2}\bigg\|\le K_G\|T\|\bigg\|\Big(\sum_{i=1}^n|x_i|^2\Big)^{1/2}\bigg\|,$$
where $K_G$ is the universal Grothendieck constant.
\end{thm}

\begin{thm}\label{Th:MVvine}
 Let $d\geq 2$ and for $\zeta\in \C$, let $\dl(\zeta)$ be defined as in \eqref{zetaC}. Then for any $p>2$ we have the
 vector-valued inequality
$$ \bigg(\int_0^{\infty}\Big (\sum_{m=0}^{\infty}\sum_{j=1}^{d(m)}\big( \sup_{N}\big|\mathbb{S}_N^{\dl(i\beta+\varepsilon)}
f_{m,j}(r) \big|\big)^2 \Big )^{\frac{p}{2}}\,dr\bigg)^{\frac{1}{p}}
\leq C(\bt) \bigg(\int_0^{\infty}\Big (\sum_{m=0}^{\infty}\sum_{j=1}^{d(m)}\big|f_{m,j}(r) \big|^2 \Big )^{\frac{p}{2}}\,dr\bigg)^{\frac{1}{p}} $$
for any sequence of functions $f_{m,j}\in L^p(\R^+, dr) $ for which the right hand side is finite. Moreover, $C(\bt)$ is of admissible growth.
\end{thm}

\begin{proof}
Observe that we can consider $\displaystyle \sup_{N}\big|\mathbb{S}_N^{\dl(i\beta+\varepsilon)}f\big|$ as a linear operator
mapping $L^p(\R^+)$ into $L^p\big(\R^+,\, l^{\infty}(\Na)\big)$. Then, appealing to Theorem \ref{Kr}, it is enough to show that this maximal operator is bounded on $L^p(\R^+, dr)$.

Consider the radial function $F$ defined by $F(x)=|x|^{-\frac{d-1}{p}}f(|x|)$ which belongs to $L^p(\R^d)$ as $f \in L^p(\R^+, dr)$. In terms of $F$ we can consider $\mathbb{S}_N^{\dl(i\beta+\varepsilon)}f$ as a radial function on $\R^d$ given by
$$ |x|^{\frac{d-1}{p}}\int_{\R^d}\big|\si_N^{\dl(i\beta+\varepsilon)}(x, y) \big|F(y)dy. $$
Thus the boundedness of $\sup_{N}\big| \mathbb{S}_N^{\dl(i\beta+\varepsilon)}f\big| $ on $L^p(\R^+, dr)$ follows from the fact that the maximal operator
$$  \sup_{N}\bigg(\int_{\R^d}\Big |\si_{N}^{\dl(i\beta+\varepsilon)}(x,y) \Big |\,|F(y)| \,dy \bigg)$$
is bounded on $L^p(\R^d)$, which is true by Theorem \ref{maximal}. The theorem is proved.
\end{proof}

\section{Ces\`{a}ro means for the Dunkl--Hermite and extension to complex parameters}
\label{sec:cesaroDunkl}

In this section we establish various results in several steps:
\begin{enumerate}
\item First, we recall the Mehler's formula in the Dunkl--Hermite setting, analogous to the Mehler's formula for the standard Hermite functions \eqref{eq:Mehler}.
\item In the second step, we are able to write the Ces\`{a}ro kernels for Dunkl--Hermite expansions in terms of Ces\`{a}ro kernels for Hermite expansions with shifted parameters. As a result, by using spherical coordinates, we express the Dunkl--Ces\`{a}ro means as an expansion in $h$-harmonics of linear operators defined in terms of standard Ces\`{a}ro means, see Corollary \ref{kernelexp}.
\item In turn, we show in Proposition \ref{prop:relKernels} that the standard shifted Ces\`{a}ro kernels are connected to Ces\`{a}ro kernels for Laguerre expansions.
\item Finally, we need to extend this connection to complex values of the parameters in Proposition \ref{prop:Compkernel}.
\end{enumerate}

These facts are shown in this order in the subsequent subsections. They will lead us to the proof of Theorem \ref {Th:Vvine} and, as a consequence, of Theorem \ref{main}. Since the proofs of Propositions \ref{lem:hProj}, \ref{prop:relKernels} and \ref{prop:Compkernel} are rather technical, we state the results here and we give their proofs in Section \ref{sec:propositions} separately.
\subsection{Mehler's formula for Dunkl--Hermite functions}

As stated in the introduction, generalized Hermite polynomials and generalized
Hermite functions associated with Coxeter groups were studied in \cite{RO}, where the precise definitions can be found. It is known that the Dunkl-Hermite functions $\Phi_{\mu, \K}$ are eigenfunctions of Dunkl-Hermite operator $\Dk$ with eigenvalues $(2|\mu|+2\gm+d)$. For our purposes, the most important result is the generating function identity or the \textit{Mehler's formula for the Dunkl--Hermite functions}.   For $ |w| < 1 $, one has (see \cite[Theorem 3.12]{RO})
\begin{equation}\label{eq:MehlerGen1}
\sum_{\mu \in \mathbb{N}^d}\Phi_{\mu,\K}(x)\Phi_{\mu,\K}(y)w^{|\mu |} =\frac{2}{\om_d^{\K}\,\Gm\big (\frac{d}{2}+\gm\big )}(1-w^2)^{-\frac{d}{2}-\gm}e^{-\frac{1}{2}\big (\frac{1+w^2}{1-w^2}\big )  (|x|^2+|y|^2)}E_{\K}\bigg (\frac{2wx}{1-w^2}, y \bigg ).
\end{equation}

\subsection{Ces\`{a}ro kernels for Dunkl--Hermite expansions}
In this subsection we obtain an expression for the Ces\`{a}ro kernel for the Dunkl--Hermite expansions in terms of the Ces\`{a}ro
kernel for the standard Hermite expansions, via the Funk--Hecke formula stated in Corollary \ref{eq:Hecke}.

Let $P_{j,\K}$ be the orthogonal projection described in \eqref{eq:L2expn}. Then $P_{j,\K}$ is given by the kernel operator
\begin{equation}\label{eq:projection}
 P_{j,\K}f(x)= \int_{\R^d}\Phi_{j,\K}(x,y)f(y)h_{\K}^2(y)dy
\end{equation}
where
\begin{equation*}
\Phi_{j,\K}(x, y)= \sum_{|\mu|= j}\Phi_{\mu,\K}(x)\Phi_{\mu,\K}(y).
\end{equation*}

For $\delta \geq 0$, it is clear from \eqref{eq:DunklCesaro} and \eqref{eq:projection} that the kernel of the Ces\`{a}ro means $ \sigma_{N,\K}^{\delta} $ is given by
\begin{equation}\label{eq:CDHKernel}
\si_{N,\K}^{\dl}(x,y):=\frac{1}{A_N^{\dl}}\sum_{j=0}^NA_{N-j}^{\dl}\Phi_{j,\K}(x,y).
\end{equation}

Proposition \ref{lem:hProj} below contains expressions for the $h$-harmonic coefficients of Dunkl--Ces\`{a}ro kernels in terms of $d$-dimensional and $(d+1)$-dimensional standard Ces\`{a}ro kernels. As remarked earlier,  in view of the Mehler's formula \eqref{eq:Mehler} it follows that $\Phi_{j}(rx',sy')$ is a function of $r, s$ and $u:= x' \cdot y'$. The same is true for the Ces\`{a}ro kernels and so again we will sometimes write $\Phi_{j}(r,s; u)$ instead of $ \Phi_{j}(rx',sy')$ and $\sigma_{N}^{\dl}(r,s; u)$ instead of $ \sigma_{N}^{\dl}(rx',sy')$. Let $P_m^{\lambda} $ stand for the normalized ultraspherical polynomials of type $\lambda>-\tfrac12$ and degree $m$. We have the following.

\begin{prop}[Funk--Hecke for Ces\`{a}ro--Dunkl--Hermite]\label{lem:hProj}
For $\gm$ defined as in \eqref{gamma}, let
\begin{equation*}
c_{d,\gamma} = \frac{2 \pi^{(d-1)/2}}{\Gamma(\frac{d-1}{2}+\gamma)}
\end{equation*}
and $ \lambda = \tfrac{d}{2}+\gamma-1 $.
Then, for any spherical h-harmonic $ Y_{m,\ell}$ of degree $m$, we have the identity
\begin{multline}\label{HBCDHd}
\int_{\sd}\si_{N,\K}^{\dl,d}(rx', sy')Y_{m,\ell}(y') h_{\K}^2(y')\,d\si(y') \\
=\frac{c_{d,\gm}}{A_N^{\dl}}\sum_{j=0}^{[N/2]}A_{j}^{\gm-1}
A_{N-2j}^{\delta}\left(\int_{-1}^{1}\si_{N-2j}^{\dl,d} (r,s; u)P_m^{\la}(u)(1-u^2)^{\la-1/2}\,du \right)Y_{m,\ell}(x').
\end{multline}
Moreover, we also have
\begin{multline}\label{HBCDHdp1}
\int_{\sd}\si_{N,\K}^{\dl,d}(rx', sy')Y_{m,\ell}(y') h_{\K}^2(y')\,d\si(y') \\= \frac{c_{d+1,\gm-\frac{1}{2}}}{A_N^{\dl}}\sum_{j=0}^{[N/2]}A_{j}^{\gm-\frac{3}{2}}A_{N-2j}^{\delta}\left(\int_{-1}^{1}\si_{N-2j}^{\dl, d+1}
(r,s; u)P_m^{\la}(u)(1-u^2)^{\la-1/2}\,du \right)Y_{m,\ell}(x').
\end{multline}
\end{prop}

By expanding $\sigma_{N,\K}^{\dl}f(x)$ in terms of $h$-harmonics and using Proposition \ref{lem:hProj} we can easily deduce the following.

\begin{cor}\label{kernelexp}
For any Schwartz class function $f$ we have
$$\si_{N,\K}^{\dl}f(rx')=\sum_{m=0}^N \sum_{j=1}^{d(m)}T_{N,m}^{\dl,\gm}
f_{m,j}(r) Y_{m,j}(x'),$$
where $T_{N,m}^{\dl,\gm}$ are the operators defined in \eqref{eq:alterCesaro} with kernels $K_{N,m}^{\dl,\gm}$ that can be expressed either as
\begin{equation}\label{eq:alterKernel}
K_{N,m}^{\dl,\gm}(r, s):=\frac{c_{d,\gm}}{A_N^{\dl}}\sum_{j=0}^{[N/2]}A_{j}^{\gm-1}
A_{N-2j}^{\delta}\int_{-1}^{1}\si_{N-2j}^{\dl,d}(rx',sy')P_m^{\la}(u)(1-u^2)^{\la-1/2}\,du
\end{equation}
or
\begin{equation}\label{eq:alterKernel1}
K_{N,m}^{\dl,\gm}(r, s)=\frac{c_{d+1,\gm-\frac{1}{2}}}{A_N^{\dl}}\sum_{j=0}^{[N/2]}A_{j}^{\gm-\frac{3}{2}}
A_{N-2j}^{\delta}\int_{-1}^{1}\si_{N-2j}^{\dl, d+1}(r,s;u)P_m^{\la}(u)(1-u^2)^{\la-1/2}\,du.
\end{equation}
\end{cor}

\begin{rem}
Therefore, we can express the kernels $K_{N,m}^{\dl,\gm}(r, s)$ and consequently the operators $T_{N,m}^{\dl, \gm} f$ in terms of both $d$-dimensional and $(d+1)$-dimensional Hermite Ces\`{a}ro kernels. Indeed, the expression \eqref{eq:alterKernel} comes from the use of identity \eqref{HBCDHd} in Proposition \ref{lem:hProj}, and the expression \eqref{eq:alterKernel1} is obtained in view of identity \eqref{HBCDHdp1} in Proposition \ref{lem:hProj}.
\end{rem}

\subsection{The Laguerre connection}\label{LaguerreCon}

We are going to express the kernel $ K_{N,m}^{\dl,\gm}(r, s)$ in \eqref{eq:alterKernel} in terms of Laguerre
functions. This is contained in Proposition \ref{prop:relKernels} below. This result will not be enough for our purposes;  an improved version with complex parameters is required, see Proposition \ref{prop:Compkernel}. Nevertheless, we first show the non-complex version in detail since it is interesting to understand that, philosophically, there is an underlying phenomenon of transplantation.

Let us recall some basic facts about Laguerre functions. For $\ap>-1$, let $\psi_k^\ap$ be the normalized Laguerre functions given by
$$\psi_k^{\ap}(r) = \bigg(\frac{2\Gm(k+1)}{\Gm(k+\ap+1)} \bigg)^{\frac{1}{2}}L_k^\ap(r^2)e^{-\frac{1}{2}r^2}, \quad k=0,1,\ldots, $$
where $L_{k}^\ap$ are the Laguerre polynomials of order $\ap$, see \cite[page 76]{Lebedev}.

We have the following generating function identity for Laguerre functions (see \cite[(1.1.47)]{ST}). For $|w|<1$,
\begin{equation*}
\sum_{k=0}^{\infty}\psi_{k}^{\ap}(r)\psi_{k}^{\ap}(s)w^{2k}=2(1-w^2)^{-1} (rsw)^{-\ap}e^{-\frac{1}{2}\big (\frac{1+w^2}{1-w^2} \big )(r^2+s^2)} I_{\ap}\bigg(\frac{2rsw}{1-w^2}\bigg).
\end{equation*}
Here, $I_{\alpha}$ is the modified Bessel function, see Subsection \ref{sec:ultras} for the definition and further results concerning these functions. From this identity, we easily deduce that
\begin{equation} \label{eq:generatingLag}
\sum_{k=0}^{\infty}(rs)^m \psi_{k}^{\ap+m}(r)\psi_{k}^{\ap+m}(s)w^{2k+m}=2(1-w^2)^{-1}
 e^{-\frac{1}{2}\big (\frac{1+w^2}{1-w^2} \big )(r^2+s^2)} (rsw)^{-\ap}I_{\ap+m}\bigg(\frac{2rsw}{1-w^2}\bigg).
\end{equation}

\begin{prop}\label{prop:relKernels}
For $ d \geq 2,$ let $\la=\frac{d}{2}+\gm-1.$ Then the kernel $K_{N,m}^{\dl,\gm}$ expressed either in \eqref{eq:alterKernel} or in \eqref{eq:alterKernel1} can be written as
$$K_{N,m}^{\dl,\gm}(r,s)=\frac{1}{A_{N}^{\dl}}\sum_{j=0}^{[(N-m)/2]}A_{N-m-2j}^{\delta}(rs)^m
\psi_j^{\la+m}(r)\psi_j^{\la+m}(s)$$
provided  $N\geq m$.  For other values of $N$, $K_{N,m}^{\dl,\gm}(r,s)=0$.
\end{prop}

\subsection{Extension to complex parameters}

In order to prove Theorem \ref{Th:Vvine} we need to extend our operators for complex parameters and get proper expressions for them.

First, observe that the kernels $K_{N,m}^{\dl, \gm}$ in \eqref{eq:alterKernel}, initially expressed in terms of standard Ces\'aro kernels, are also expressible in terms of Laguerre functions $\psi_k^{\la+m}$ by Proposition \ref{prop:relKernels}. Then, since Laguerre functions of complex parameters are well defined, the kernels $K_{N,m}^{\dl, \gm}$ make sense even if $\dl$ and $\gm$ are complex. For $\zeta \in \C$ let us define
\begin{equation}\label{lambdaC}
\la(\zeta):= \frac{d}{2}+\zeta-1.
\end{equation}
Fix $\varepsilon  >0$ and consider the sequence of operators
\begin{equation}\label{TNeps}
\mathcal{T}_{N,m}^{\varepsilon,\zeta}f(r):= \int_0^{\infty}\mathcal{K}_{N,m}^{ \varepsilon,\zeta}(r,s)f(s)\,ds
\end{equation}
defined for $f \in L^p(\R^+, dr)$ with kernel
\begin{equation}\label{eq:Kcal}
\mathcal{K}_{N,m}^{ \varepsilon,\zeta}(r,s):=r^{\frac{2\dl(\zeta)}{p}}s^{\frac{2\dl(\zeta)}{p'}}
K_{N,m}^{\dl(\zeta+\varepsilon),\zeta}(r,s)
\end{equation}
where $p>2 $. Then, we can go back and prove that the kernels $\mathcal{K}_{N,m}^{\ve,\zeta}(r,s)$, that  are now defined in terms of Laguerre functions of complex parameters, can be expressed also in terms of the $d$-dimensional and $(d+1)$-dimensional kernels of Ces\`{a}ro means involving parameters with certain values of $\zeta$. This result is contained in Proposition \ref{prop:Compkernel}.

Observe that the operators $\mathcal{T}_{N,m}^{\varepsilon,\zeta}$ can be understood as ``extensions'', for complex parameters, of the operators $T_{N,m}^{\dl,\gm}$. In turn, the kernels $\mathcal{K}_{N,m}^{ \varepsilon,\zeta}$ can be understood as ``extensions'', for complex parameters, of the kernels $K_{N,m}^{\dl,\gm}$.

\begin{prop}\label{prop:Compkernel}
Let  $\displaystyle C_{d}=\frac{4\pi^{(d-1)/2}}{\Gm((d-1)/2)}$. For any $\beta\in \R$ we have
\begin{multline}\label{nuc1}
\mathcal{K}_{N,m}^{\varepsilon,i\beta}(r, s)=\frac{C_d}{A_N^{\dl(i\bt+\ve)}}
\sum_{k=0}^N\binom{\ve/2}{N-k}\\
\times \sum_{j=0}^{[k/2]} A_{j}^{\frac{\ve}{2}+i\bt-1}A_{k-2j}^{\dl(i\beta+\frac{\ve}{2})}
 r^{\frac{2\dl(i\beta)}{p}}s^{\frac{2\dl(i\beta)}{p'}} \int_{-1}^{1}\si_{k-2j}^{\dl(i\beta+\frac{\ve}{2}),d}(r,s;u)Q_{m}^{\la(i\beta)}(u)\,du
\end{multline}
 and
\begin{multline} \label{nuc2}
\mathcal{K}_{N,m}^{\varepsilon,\frac12+i\beta}(r, s)=\frac{C_{d+1}}{A_N^{\dl(\frac12+i\bt+\ve)}}\sum_{k=0}^N
\binom{\ve/2}{N-k}\\
\times \sum_{j=0}^{[k/2]} A_{j}^{\tfrac{\ve}{2}+i\bt-1}A_{k-2j}^{\dl(\frac12+i\beta+\frac{\ve}{2})}
 r^{\frac{2\dl(\frac12+i\beta)}{p}} s^{\frac{2\dl(\frac12+i\beta)}{p'}}\int_{-1}^{1}\si_{k-2j}^{\dl(\frac12+i\beta+\frac{\ve}{2}),
d+1}(r,s;u)Q_{m}^{\la(\frac12+i\beta)}(u)\,du.
\end{multline}
Here, the functions $Q_m^{\lambda(\zeta)}(u)$ are the ones defined in \eqref{eq:Q}.
\end{prop}

\section{Mixed norm estimates for the Ces\`{a}ro means: proofs of Theorems \ref{main} and \ref{Th:Vvine}}
\label{subsec:Mainproofs}

In this Section we prove the main theorems. Theorem \ref{main} easily follows from Theorem \ref{Th:Vvine}. In its turn, the proof of Theorem \ref{Th:Vvine} needs the ingredients from the previous sections, and the different operators and their extensions with complex parameters play a role. Since the quantity of operators and the corresponding kernels is large and the notation cumbersome, we present the Table \ref{tab:notation} with a summary of notation and references that will make the reading of this section more comfortable.

{{\begin{table*}[htbp]
\centering
\caption{Summary of operators and kernels}
\newcommand{\vs}{\rule[-6pt]{0pt}{20pt}\ignorespaces}
\begin{tabular}{|c c|}
\hline
\vs $T_{N,m}^{\dl,\gm}$& operator defined in \eqref{eq:alterCesaro} \\
\hline
\vs $K_{N,m}^{\dl,\gm}$& kernel of $T_{N,m}^{\dl,\gm}$, defined  in \eqref{eq:alterKernel} or \eqref{eq:alterKernel1} \\
\hline
\vs $\mathcal{T}_{N,m}^{\varepsilon,\zeta}$ & operator defined in \eqref{TNeps}\\
\hline
\vs $\mathcal{K}_{N,m}^{ \varepsilon,\zeta}$& kernel of $\mathcal{T}_{N,m}^{\varepsilon,\zeta}$, related to $K_{N,m}^{\dl,\gm}$ by \eqref{eq:Kcal}\\
\hline
\vs $\mathbb{S}_N^{\dl}$& operator defined in \eqref{eq:supremeOp} in terms of the standard Ces\`{a}ro means $\si_N^{\dl}$ \\
\hline

\end{tabular}

\label{tab:notation}
\end{table*}
}}

\begin{proof}[Proof of Theorem \ref{main}]
In view of Corollary \ref{kernelexp} we have that
\begin{align*}
\|\si_{N,\K}^{\dl}f\|_{(p,2)}&=\bigg(\int_0^\infty\bigg (\int_{\sd}|\si_{N,\K}^{\dl}f(rx')|^2h_{\K}^2(x')\,d\si(x')
\bigg )^{\frac{p}{2}}r^{d+2\gm-1}\,dr \bigg)^{\frac{1}{p}}\\
&=\bigg(\int_0^\infty\bigg (\int_{\sd}\bigg|\sum_{m=0}^\infty\sum_{j=1}^{d(m)}|T_{N,m}^{\dl,\gm}
f_{m,j}(r)Y_{m,j}(x')\bigg|^2h_{\K}^2(x')\,d\si(x') \bigg )^{\frac{p}{2}}r^{d+2\gm-1}\,dr \bigg)^{\frac{1}{p}}\\
&=\Big\|\Big(\sum_{m=0}^\infty\sum_{j=1}^{d(m)}|T_{N,m}^{\dl,\gm}f_{m,j}(r)|^2\Big)^{1/2}
\Big\|_{L^p(\R^+, r^{d+2\gamma-1}dr)}.
\end{align*}
Then, with Theorem \ref{Th:Vvine}, we can conclude.
\end{proof}

Therefore, it remains to prove Theorem \ref{Th:Vvine}. Given  sequences of functions $(f_{m,j}) $ and $(g_{m,j}) $ such
that
$$\int_0^\infty \Big ( \sum_{m=0}^\infty\sum_{j=1}^{d(m)}|f_{m,j}(r)|^2 \Big )^{\frac{p}{2}}r^{d+2\gm-1}\,dr=
\int_0^\infty \Big ( \sum_{m=0}^\infty\sum_{j=1}^{d(m)}|g_{m,j}(r)|^2 \Big )^{\frac{p'}{2}}r^{d+2\gm-1}\,dr=1 $$
we define $\widetilde{f}_{m,j}(r)=f_{m,j}(r)r^{\frac{2\dl(\gm)}{p}} $ and $\widetilde{g}_{m,j}(r)=g_{m,j}(r)r^{\frac{2\dl(\gm)}{p'}}$, where $\dl(\gm)$ is as in \eqref{zetaC}. Then it follows that
$$\int_0^\infty \Big ( \sum_{m=0}^\infty\sum_{j=1}^{d(m)}|\widetilde{f}_{m,j}(r)|^2 \Big )^{\frac{p}{2}}
\,dr=\int_0^\infty \Big ( \sum_{m=0}^\infty\sum_{j=1}^{d(m)}|\widetilde{g}_{m,j}(r)|^2 \Big )^{\frac{p'}{2}}\,dr=1. $$
We consider the function $F_N(\zeta)$ on the strip $0\leq \Re(\zeta)\leq \frac12$ defined by
$$  F_N(\zeta):=\int_0^{\infty}\bigg(\sum_{m=0}^\infty\sum_{j=1}^{d(m)}\mathcal{T}_{N,m}^{\ve,\zeta} \widetilde{f}_{m,j}(r)\overline{\widetilde{g}_{m,j}}(r) \bigg)dr.$$
It is clear that when $ f_{m,j} $ are compactly supported on $ \R^+ $, $F_N(\zeta)$ is a holomorphic function in the interior
of the strip $0\leq \Re(\zeta)\leq \frac12$ and continuous up to the boundary. We will prove that $F_N(\Re(\zeta))$ is bounded in the strip. In order to do that, we need the following variant of three lines lemma proved in \cite[Ch.\ V, Lemma 4.2]{SW}.
\begin{lem}[Stein-Weiss]
\label{SW}Suppose $F$ is a function defined and continuous  on the unit strip   $\mathcal{S}=\{z\in \C: 0\leq \Re(z)\leq 1\} $
 that is analytic in the interior of $\mathcal{S}$ satisfying
$$\sup_{x+iy \in \mathcal{S}}e^{-a|y|}\log{|F(x+iy)|}< \infty$$
for some $a< \pi $. Then
\begin{equation} \label{eq:3ll}
 \log{|F(x)|}\leq \frac12 \sin{\pi x}\; \int_{-\infty}^{\infty} \bigg(\frac{\log{|F(iy)|}}{\cosh{\pi y}-\cos{\pi x}}
 \:+\:\frac{\log{|F(1+iy)|}}{\cosh{\pi y}+\cos{\pi x}}\bigg)\, dy
 \end{equation} whenever $0< x< 1$.
 \end{lem}

\begin{prop}\label{prop:estimate}
Let $\zeta:=\gm+i\beta$, with $0 \leq \gm \leq 1/2$ and $\beta\in \R$. We have
$$ |F_N(i\bt)|\leq C_0(\bt) \qquad \hbox{and} \qquad |F_N(\tfrac12+i\bt)|\leq C_1(\bt) $$
where $ C_0(\bt)$ and $C_1(\bt) $ are independent of $N$ and of admissible growth. Moreover,
$$|F_N(\gm)|\le C, \quad \hbox{ for } \quad 0 \leq \gm \leq1/2. $$
\end{prop}

\begin{proof}
In view of \eqref{nuc1} in Proposition \ref{prop:Compkernel}, the operator $\mathcal{T}_{N,m}^{\ve,i\bt}$ is given by
\begin{multline*}
\mathcal{T}_{N,m}^{\ve,i\bt}f(r)=\frac{C_d}{A_N^{\dl(i\bt+\ve)}} \sum_{k=0}^N\binom{\ve/2}{N-k}\sum_{j=0}^{[k/2]}
A_{j}^{\frac{\ve}{2}+i\bt-1}A_{k-2j}^{\dl(i\beta+\frac{\ve}{2})}\\ \times
\int_0^{\infty}r^{\frac{2\dl(i\beta)}{p}}s^{\frac{2\dl(i\beta)}{p'}} \int_{-1}^{1}\si_{k-2j}^{\dl(i\beta+\frac{\ve}{2}),d}(r,s;u)Q_{m}^{\la(i\beta)}(u) \,duf(s)\,ds.
\end{multline*}
The operators $\mathcal{T}_{N,m}^{\ve,i\bt}$ can be bounded in terms of the operators $\mathbb{S}_k^{\dl(i\bt+\ve/2),d}$.
Indeed, taking into account the fact that $\sum_{j=0}^{\infty}\big|\binom{\ve/2}{j}\big|< \infty $, we have
$$ |\mathcal{T}_{N,m}^{\ve,i\beta}\widetilde{f}_{m,j}(r) | \leq C_0(\beta)\; \sup_{N}|\mathbb{S}_N^{\dl(i\beta+\ve /2),d} \widetilde{f}_{m,j}(r)|$$
provided we have the estimate
$$\frac{C_{d}}{\big|A_{N}^{\dl(i\beta+\ve)}\big|}  \sum_{j=0}^{[k/2]} \big|A_{j}^{\ve/2+i\bt-1}A_{k-2j}^{\dl(i\beta+\ve/2)}\big| \leq C_0(\beta) $$
for all $0\leq k\leq N $, $N=1,2,3,\cdots$. But this can be proved similarly as in the case of Lemma
\ref{lem:bin.est}, and hence we leave the details to the reader.

With the above observations, we obtain, for $p>2$,
\begin{align*}
|F_N(i\bt)|&\le \int_0^{\infty}\bigg(\sum_{m=0}^\infty\sum_{j=1}^{d(m)}|\mathcal{T}_{N,m}^{\ve,i\beta} \widetilde{f}_{m,j}(r)|^2\bigg)^{1/2}\bigg(\sum_{m=0}^\infty \sum_{j=1}^{d(m)}|\overline{\widetilde{g}_{m,j}}(r)|^2\bigg)^{1/2}dr\\
&\le C_0(\beta)\int_0^{\infty}\bigg(\sum_{m=0}^\infty\sum_{j=1}^{d(m)} (\sup_{N}|\mathbb{S}_N^{\dl(i\beta+\ve /2), d} \widetilde{f}_{m,j}(r)|)^2\bigg)^{1/2} \bigg(\sum_{m=0}^\infty\sum_{j=1}^{d(m)}|\overline{\widetilde{g}_{m,j}}(r)|^2\bigg)^{1/2}dr\\
&\le C_0(\beta)\bigg(\int_0^{\infty}\bigg(\sum_{m=0}^\infty \sum_{j=1}^{d(m)}(\sup_{N}|\mathbb{S}_N^{\dl(i\beta+\ve /2), d} \widetilde{f}_{m,j}(r)|)^2\bigg)^{p/2}dr\bigg)^{1/p}\\
&\qquad \qquad \times \bigg(\int_0^{\infty} \bigg(\sum_{m=0}^\infty\sum_{j=1}^{d(m)}|\overline{\widetilde{g}_{m,j}}(r)|^2\bigg)^{p'/2}dr\bigg)^{1/p'}\\
&\le C_0(\beta) \bigg(\int_0^{\infty} \bigg(\sum_{m=0}^\infty\sum_{j=1}^{d(m)}|\widetilde{f}_{m,j}(r)|^2\bigg)^{p/2}dr\bigg)^{1/p}=C_0(\beta),
\end{align*}
where in the last inequality we applied Theorem \ref{Th:MVvine}. Similarly, by \eqref{nuc2} in Proposition \ref{prop:Compkernel}, we get the estimate
$$|\mathcal{T}_{N,m}^{\ve,1/2+i\beta}\tilde{f}_{m,j}(r)| \leq C_1(\beta) \sup_{N}|\mathbb{S}_N^{\dl(i\beta+\ve /2), d+1} \tilde{f}_{m,j}(r)|.$$
Then, an analogous reasoning with proper modifications, leads us to $|F_N(\tfrac12+i\bt)|\leq C_1(\bt)$.

Observe that with the estimates just proven, we apply Lemma \ref{SW} to our function $F_N$, so that right hand side of \eqref{eq:3ll} is finite, therefore ensuring the boundedness of $F_N(\gm)$, for $0<\gm<1/2$.
\end{proof}

We recall the following result that corresponds with an idea of J. L. Rubio de Francia tacitly contained in \cite[Remark (a)]{RF}. Here we stated it as a theorem.

\begin{thm}[Rubio de Francia] \label{th:RdF}
Let $T:L^p(\R^d)\to L^p(\R^d)$ be a bounded linear operator which is rotation invariant, i.e. it commutes with the action of the rotation group $ SO(d) $ on $ L^p(\R^d).$ Then, $T$ is also bounded in $L^{p}(\ell^2)$.
\end{thm}

We are ready to prove Theorem \ref{Th:Vvine}.

\begin{proof}[Proof of Theorem \ref{Th:Vvine}]
We first observe that when $\gm=0$ the Dunkl--Hermite expansion reduces to the standard Hermite expansion and hence Theorem
\ref{Th:Vvine} is true. Indeed, we have $L^p(\R^d)$ boundedness of $\si_N^{\dl}$ for any $\dl>\frac{d-1}{2} $, $1\leq p <\infty $
(see \cite{ST1}). On the other hand, it is easy to check that the Ces\`{a}ro kernels associated to standard Hermite expansions are
rotation invariant. Then, from Theorem \ref{th:RdF}, we immediately get the desired vector-valued extension. When $ \gm
= \frac12$, the same reasoning holds by considering Ces\`{a}ro means $\si_N^{\dl, d+1}$. Hence we can restrict ourselves to the
case $0<\gm<\frac12$.

Given $ \dl > \frac{d+2\gm-1}{2} $ we can choose $ \ve > 0 $ so that $ \dl = \frac{d+2\gm+2\ve-1}{2} = \dl(\gamma+\ve)$, which is defined  in \eqref{zetaC}. It suffices to prove that, for $ p >
2 $,
 $$\Big\|\Big(\sum_{m=0}^\infty\sum_{j=1}^{d(m)}|T_{N,m}^{\dl(\gm+\ve),\gm}f_{m,j}(r)|^2\Big)^{1/2}
\Big\|_{L^p(\R^+, r^{d+2\gamma-1}dr)}\le C\Big\|\Big(\sum_{m=0}^\infty\sum_{j=1}^{d(m)}|f_{m,j}(r)|^2\Big)^{1/2}\Big\|_{L^p(\R^+,  r^{d+2\gamma-1}dr)}.$$
We can write
\begin{align*}
\int_0^{\infty}\bigg(\sum_{m=0}^\infty&\sum_{j=1}^{d(m)}T_{N,m}^{\dl(\gm+\ve),\gamma} f_{m,j}(r)\overline{g_{m,j}}(r) \bigg)r^{d+2\gamma-1}\,dr\\
&=\int_0^{\infty}\bigg(\sum_{m=0}^\infty\sum_{j=1}^{d(m)}\int_0^{\infty}K_{N,m}^{\dl(\gm+\ve),\gm}(r,s)f_{m,j}(s)s^{d+2\gamma-1}\,ds\, \overline{g_{m,j}}(r) \bigg)r^{d+2\gamma-1}\,dr\\
&=\int_0^{\infty}\bigg(\sum_{m=0}^\infty\sum_{j=1}^{d(m)}\int_0^{\infty}r^{\frac{2\dl(\gm)}{p}}s^{\frac{2\dl(\gm)}{p'}}
K_{N,m}^{\dl(\gm+\ve),\gm}(r,s)f_{m,j}(s)s^{\frac{2\dl(\gm)}{p}}\,ds\,\overline{g_{m,j}}(r) \bigg)r^{\frac{2\dl(\gm)}{p'}}\,dr\\
&=\int_0^{\infty}\bigg(\sum_{m=0}^\infty\sum_{j=1}^{d(m)}\int_0^{\infty} \mathcal{K}_{N,m}^{\ve,\gm}(r,s)\widetilde{f}_{m,j}(s)\,ds\, \overline{\widetilde{g}_{m,j}}(r) \bigg)\,dr\\
&=\int_0^{\infty}\bigg(\sum_{m=0}^\infty\sum_{j=1}^{d(m)}T_{N,m}^{\ve,\gm} \widetilde{f}_{m,j}(r)\overline{\widetilde{g}_{m,j}}(r) \bigg)dr=F_N(\gm),
\end{align*}
and the last one is bounded, by Proposition \ref{prop:estimate}. By an argument of duality, the estimate is valid for all $1<p<\infty$.
\end{proof}

\section{Proofs of Propositions \ref{lem:hProj}, \ref{prop:relKernels} and \ref{prop:Compkernel}}
\label{sec:propositions}
\begin{proof}[Proof of Proposition \ref{lem:hProj}]
By integrating Mehler's formula for the Dunkl--Hermite functions \eqref{eq:MehlerGen1} against $ Y_{m,\ell}(y') $
and using the Funk--Hecke formula in Corollary \ref{eq:Hecke} we get
\begin{multline} \label{eq:MehlerGenInt}
\sum_{N=0}^\infty w^N \int_{\sd}\Phi^{(d)}_{N,\K}(rx', sy')Y_{m,\ell}(y')h_{\K}^2(y')\,d\si(y')\\
= \frac{2(1-w^2)^{-d/2-\gm}}{\sqrt{\pi}\Gm(\frac{d-1}{2}+\gm)}e^{-\frac{1}{2}
\big (\frac{1+w^2}{1-w^2} \big )(r^2+s^2)}\bigg(\int_{-1}^1e^{\frac{2rsw}{1-w^2}u}P_m^{\la}(u)(1-u^2)^{\la-\frac12}\,du\bigg) Y_{m,\ell}(x').
\end{multline}
Now, comparing the right hand side of the above with the Mehler's formula for the Hermite expansions \eqref{eq:Mehler}, we conclude that
\begin{multline}\label{eq:MehlerInt}
\sum_{N=0}^\infty w^N \int_{\sd}\Phi^{(d)}_{N,\K}(rx', sy')Y_{m,\ell}(y')h_{\K}^2(y')\,d\si(y')\\
=\frac{2 \pi^{(d-1)/2}(1-w^2)^{-\gm}}{\Gm\big(\frac{d-1}{2}+\gm \big)}\sum_{N=0}^{\infty}\bigg(\int_{-1}^1
 \Phi^{(d)}_{N}(r,s;u)P_m^{\la}(u)(1-u^2)^{\la-\frac12}\,du\bigg)w^N\, Y_{m,\ell}(x').
\end{multline}
By multiplying the identity \eqref{eq:MehlerInt} by  $(1-w)^{-\dl-1} $  and using the expansion
$(1-w)^{-\dl-1}=\sum_{N=0}^{\infty}A_N^{\dl}w^N $, we obtain
\begin{align*}
\sum_{N=0}^\infty &w^N \sum_{j=0}^N A_{N-j}^{\dl} \left(\int_{\sd}\Phi^{(d)}_{j,\K}(rx',
sy')Y_{m,\ell}(y')h_{\K}^2(y')\,d\si(y')\right)\\
&=c_{d,\gm}(1-w^2)^{-\gm}\sum_{N=0}^{\infty}w^N\sum_{j=0}^N A_{N-j}^{\dl}\bigg(\int_{-1}^1\Phi^{(d)}_{j}(r,s;u)P_m^{\la}(u)(1-u^2)^{\la-\frac12}\,du \bigg) Y_{m,\ell}(x')\\
&=c_{d,\gm}(1-w^2)^{-\gm}\sum_{N=0}^{\infty}w^N A_N^{\dl} \bigg(\int_{-1}^1\si_{N}^{\dl,d} (r,s;u)P_m^{\la}(u)(1-u^2)^{\la-1/2}\,du \bigg)Y_{m,\ell}(x').
\end{align*}
Since $\displaystyle (1-w^2)^{-\gm}= \sum_{N=0}^{\infty} A_N^{\gm-1} w^{2N}$, the above can be written as
\begin{align*}
\sum_{N=0}^\infty &w^N \sum_{j=0}^N A_{N-j}^{\dl}\bigg(\int_{\sd}\Phi^{(d)}_{j,\K}(rx',
sy')Y_{m,\ell}(y')h_{\K}^2(y')\,d\si(y')\bigg)\\
&=c_{d,\gm}\sum_{N=0}^{\infty}w^N\sum_{j=0}^{[N/2]}A_j^{\gm-1}A_{N-2j}^{\dl} \bigg(\int_{-1}^1\si_{N-2j}^{\dl,d} (r,s;u)P_m^{\la}(u)(1-u^2)^{\la-1/2}\,du\bigg)Y_{m,\ell}(x').
\end{align*}
Comparing the coefficients of $w^N$ on both sides in the above and in view of \eqref{eq:CDHKernel} we obtain \eqref{HBCDHd}.

Observe that the right hand side of  \eqref{eq:MehlerGenInt} remains the same if we replace $ d $ by $ d+1 $ and $ \gamma $ by $ \gamma-\frac{1}{2}$. Consequently, in view of Mehler's formula for the Hermite expansions \eqref{eq:Mehler}, the right hand side of  \eqref{eq:MehlerGenInt} is also equal to
\begin{equation*}
\frac{2 \pi^{d/2} (1-w^2)^{-\gm+\frac{1}{2}}}{\Gm\big(\frac{d-1}{2}+\gm \big)}\sum_{N=0}^{\infty}\bigg(\int_{-1}^1
 \Phi_{N}^{(d+1)}(r,s;u)P_m^{\la}(u)(1-u^2)^{\la-\frac12}\,du\bigg)w^N\, Y_{m,\ell}(y').
\end{equation*}
Therefore, by repeating an analogous reasoning, we also arrive at \eqref{HBCDHdp1}.

This completes the proof of the proposition.
\end{proof}

\begin{proof}[Proof of Proposition \ref{prop:relKernels}]
We integrate Mehler's formula for $\Phi_N(x,y)$ in \eqref{eq:Mehler} against the function $P_m^{\la}(u)(1-u^2)^{\la-1/2}$ over $(-1,1)$, and apply Lemma \ref{lem:UltraBessel} with $z=\frac{2w}{1-w^2}rs$, so that we get
\begin{align*}
\sum_{N=0}^\infty w^N\int_{-1}^1&\Phi_N(r,s;u)P_m^{\la}(u)(1-u^2)^{\la-1/2}\,du\\
&=\pi^{-\frac{d}{2}} (1-w^2)^{-\frac{d}{2}}e^{-\frac{1}{2}\big (\frac{1+w^2}{1-w^2} \big )(r^2+s^2)}\int_{-1}^1P_m^{\la}(u)(1-u^2)^{\la-1/2}e^{zu}\,du\\
&=\pi^{-(d-1)/2}\Gm\big(\la+1/2) (1-w^2)^{-\frac{d}{2}}e^{-\frac{1}{2} \big (\frac{1+w^2}{1-w^2} \big )(r^2+s^2)}\frac{I_{\la+m}(z)}{(z/2)^{\la}}.
\end{align*}
Comparing this with the generating function identity  for the Laguerre functions \eqref{eq:generatingLag} we obtain
\begin{multline}\label{eq:equalingHL}
\sum_{N=0}^\infty w^N\int_{-1}^1\Phi_N(r, s; u)P_m^{\la}(u)(1-u^2)^{\la-1/2}\,du\\
=\frac{\Gm(\la+1/2)}{2\pi^{(d-1)/2}} (1-w^2)^{\gm}(wrs)^m\sum_{N=0}^{\infty}\psi_{N}^{\la+m}(r)\psi_{N}^{\la+m}(s)w^{2N}.
\end{multline}

On the one hand, multiplying $\sum_{N=0}^\infty w^N\Phi_N(r,s;u)$ by $(1-w)^{-\dl-1}$, we have
\begin{equation}\label{eq:2}
(1-w)^{-\delta-1}\Big(\sum_{N=0}^\infty w^N\Phi_N(r,s;u)\Big)= \sum_{N=0}^{\infty}\Big(\sum_{j=0}^{N}A_{N-j}^{\dl}\Phi_{j}(r,s;u)\Big)w^N= \sum_{N=0}^{\infty}A_{N}^{\dl}\si_{N}^{\dl}(r,s;u)w^N.
\end{equation}
On the other hand, we also have
\begin{equation}\label{eq:1}
(1-w)^{-\delta-1}\bigg(\sum_{N=0}^{\infty} \psi_{N}^{\la+m}(r)\psi_{N}^{\la+m}(s)w^{2N}\bigg) =\sum_{N=0}^{\infty}\Big(\sum_{j=0}^{[N/2]} A_{N-2j}^{\delta} \psi_{j}^{\la+m}(r)\psi_{j}^{\la+m}(s)\Big)w^{N}.
\end{equation}
Therefore, in view of \eqref{eq:2} and \eqref{eq:1}, from \eqref{eq:equalingHL} we get
\begin{multline*}
\sum_{N=0}^{\infty}A_{N}^{\delta}\Big(\int_{-1}^{1}\si_{N}^{\dl}(r,s;u)P_m^{\la}(u) (1-u^2)^{\la-1/2}\,du\Big)w^{N}\\
=\frac{\Gm(\la+1/2)}{2\pi^{(d-1)/2}} (1-w^2)^{\gm}(wrs)^m \sum_{N=0}^{\infty}\Big(\sum_{j=0}^{[N/2]}A_{N-2j}^{\delta}
\psi_{j}^{\la+m}(r)\psi_{j}^{\la+m}(s)\Big)w^{N}.
\end{multline*}
Now, we multiply both sides by  $(1-w^2)^{-\gm}$ and rearrange to get
\begin{align*}
\sum_{N=0}^{\infty}\Big(\sum_{j=0}^{[N/2]}A_{j}^{\gm-1}& A_{N-2j}^{\delta}\int_{-1}^{1}\si_{N-2j}^{\dl}
(r,s;u)P_m^{\la}(u)(1-u^2)^{\la-1/2}\,du\Big)w^{N}\\
\notag &=\frac{\Gm(\la+1/2)}{2\pi^{(d-1)/2}} \sum_{N=0}^{\infty}\bigg(\sum_{j=0}^{[N/2]}A_{N-2j}^{\delta}(rs)^m
\psi_{j}^{\la+m}(r)\psi_{j}^{\la+m}(s)\bigg)w^{N+m}\\
&=\frac{\Gm(\la+1/2)}{2\pi^{(d-1)/2}} \sum_{N=m}^{\infty}\bigg(\sum_{j=0}^{[(N-m)/2]}A_{N-m-2j}^{\delta}(rs)^m
\psi_{j}^{\la+m}(r)\psi_{j}^{\la+m}(s)\bigg)w^{N}.
\end{align*}
Equating  the coefficients of $w^{N}$ on both sides, by \eqref{eq:alterKernel} we see that
$$K_{N,m}^{\dl,\gm}(r,s)=\frac{1}{A_{N}^{\dl}}\sum_{j=0}^{[(N-m)/2]}A_{N-m-2j}^{\delta}(rs)^m
\psi_j^{\la+m}(r)\psi_j^{\la+m}(s)$$
for $N\geq m$. It is clear that  $K_{N,m}^{\dl,\gm}(r,s)=0 $ for $ N < m$. This completes the proof of the proposition.

Repeating the same procedure, but starting with Mehler's formula \eqref{eq:Mehler} for $\Phi_N^{(d+1)}(x,y)$, $x,y\in \R^{d+1}$, we obtain the proposition for the kernel $K_{N,m}^{\dl,\gm}(r,s)$ expressed as in \eqref{eq:alterKernel1}.
\end{proof}

\begin{proof}[Proof of Proposition \ref{prop:Compkernel}]

Let $Q_m^{\lambda(\zeta)}(u)$ be as in Lemma \ref{lem:Q}, where $\lambda(\zeta)$ is taken as in \eqref{lambdaC}. We
integrate Mehler's formula for $\Phi_N(x,y)$ in \eqref{eq:Mehler} against $Q_m^{\la(i\beta)}(u)$, and use \eqref{eq:BesselUltra} with $z=\frac{2w}{1-w^2}rs$, so that
\begin{align*}
\sum_{N=0}^\infty w^N\int_{-1}^1\Phi_N(r,s;u)&Q_m^{\la(i\beta)}(u)\,du\\
&=\pi^{-\frac{d}{2}}(1-w^2)^{-\frac{d}{2}}e^{-\frac{1}{2} \big (\frac{1+w^2}{1-w^2} \big )(r^2+s^2)}\int_{-1}^1Q_m^{\la(i\beta)}(u)e^{zu}\,du\\
&=\frac{\Gm((d-1)/2)}{2\pi^{(d-1)/2}} (1-w^2)^{-\frac{d}{2}}e^{-\frac{1}{2} \big (\frac{1+w^2}{1-w^2} \big )(r^2+s^2)}\frac{I_{\la(i\beta)+m}(z)}{(z/2)^{\la(i\beta)}}.
\end{align*}
From this and \eqref{eq:generatingLag} we have
\begin{multline}\label{eq:4}
\sum_{N=0}^\infty w^N\int_{-1}^1\Phi_N(r,s;u)Q_m^{\la(i\beta)}(u)\,du\\
=\frac{\Gm((d-1)/2)}{4\pi^{(d-1)/2}}(1-w^2)^{i\beta}\sum_{N=0}^{\infty}(rs)^m \psi_{N}^{\la(i\beta)+m}(r)\psi_{N}^{\la(i\beta)+m}(s)w^{2N+m}.
\end{multline}
Now we follow a similar procedure used in the proof of Proposition \ref{prop:relKernels} but with modifications. Namely, instead of multiplying both sides by $(1-w^2)^{-\gm}(1-w)^{-\dl-1} $ we multiply by the factor $(1-w^2)^{-i\bt}(1-w)^{-\dl(i\bt+\ve)-1} $, which we rewrite as $(1-w^2)^{-\tfrac{\ve}{2}-i\bt}(1-w)^{-\dl(i\bt+\frac{\ve}{2})-1}(1+w)^{\frac{\ve}{2}}$, where $\dl(\zeta)$ is as in \eqref{zetaC}. Then, left hand side of \eqref{eq:4} delivers
\begin{multline*}
(1-w^2)^{-\tfrac{\ve}{2}-i\bt}(1-w)^{-\dl(i\bt+\frac{\ve}{2})-1}(1+w)^{\frac{\ve}{2}}\sum_{N=0}^\infty
w^N\int_{-1}^1 \Phi_N(r,s;u)Q_m^{\la(i\beta)}(u)\,du\\
=(1+w)^{\frac{\ve}{2}}\sum_{N=0}^{\infty}w^N \Big(\sum_{j=0}^{[N/2]} A_{j}^{\frac{\ve}{2}+i\bt-1}A_{N-2j}^{\dl(i\beta+\frac{\ve}{2})} \int_{-1}^{1}\si_{N-2j}^{\dl(i\beta+\frac{\ve}{2})}(r,s;u)Q_{m}^{\la(i\beta)}(u)\,du\Big)
\end{multline*}
On the other hand, from the right hand side of \eqref{eq:4} we get
\begin{align*}
C_d^{-1}(1-w^2)^{-\tfrac{\ve}{2}}&(1-w)^{-\dl(i\bt+\tfrac{\ve}{2})-1}(1+w)^{\tfrac{\ve}{2}}\sum_{N=0}^{\infty}(rs)^m
 \psi_{N}^{\la(i\beta)+m}(r)\psi_{N}^{\la(i\beta)+m}(s)w^{2N+m}\\
 &=C_d^{-1}(1-w)^{-\dl(i\bt+\ve)-1}\sum_{N=0}^{\infty}(rs)^m \psi_{N}^{\la(i\beta)+m}(r)\psi_{N}^{\la(i\beta)+m}(s)w^{2N+m}\\
 &=C_d^{-1} \sum_{N=m}^{\infty}\bigg(\sum_{j=0}^{[(N-m)/2]}A_{N-m-2j}^{\dl(i\bt+\ve)}(rs)^m \psi_{j}^{\la(i\bt)+m}(r)\psi_{j}^{\la(i\bt)+m}(s)\bigg)w^{N}\\
 &=C_d^{-1}\sum_{N=m}^{\infty}w^N A_N^{\dl(i\beta+\ve)}K_{N,m}^{\dl(i\beta+\ve), i\bt}(r, s),
\end{align*}
where the last equality is true  in view of Proposition \ref{prop:relKernels}. Altogether, we have
\begin{multline*}
\sum_{N=0}^{\infty}w^N \Big(\sum_{j=0}^{[N/2]} A_{j}^{\tfrac{\ve}{2}+i\bt-1}A_{N-2j}^{\dl(i\beta+\tfrac{\ve}{2})}
\int_{-1}^{1}\si_{N-2j}^{\dl(i\beta+\tfrac{\ve}{2})}(r,s;u)Q_{m}^{\la(i\beta)}(u)\,du\Big)\\
=\frac{(1+w)^{-\tfrac{\ve}{2}}}{C_d}\sum_{N=m}^{\infty}w^N A_N^{\dl(i\beta+\ve)}K_{N,m}^{\dl(i\beta+\ve), i\bt}(r, s).
\end{multline*}
Multiplying both sides above by  $r^{\frac{2\dl(i\beta)}{p}}s^{\frac{2\dl(i\beta)}{p'}}(1+w)^{\ve/2} $ and using $
(1+w)^{\ve/2}=\sum_{j=0}^{\infty}\binom{\ve/2}{j}w^j $, we see that
\begin{multline*}
\sum_{N=0}^{\infty}w^N\bigg(\sum_{k=0}^N\binom{\ve/2}{N-k}\sum_{j=0}^{[k/2]}
 A_{j}^{\tfrac{\ve}{2}+i\bt-1}A_{k-2j}^{\dl(i\beta+\tfrac{\ve}{2})}  r^{\frac{2\dl(i\beta)}{p}}s^{\frac{2\dl(i\beta)}{p'}}
\int_{-1}^{1}\si_{k-2j}^{\dl(i\beta+\tfrac{\ve}{2})}(r,s;u)Q_{m}^{\la(i\beta)}(u)\,du \bigg)\\
=\frac{1}{C_d}\sum_{N=m}^{\infty}w^N A_N^{\dl(i\bt+\ve)}\mathcal{K}_{N,m}^{\varepsilon,i\beta}(r, s).
\end{multline*}
Equating the coefficients of $w^N$ on both sides, we get \eqref{nuc1}.

Similarly, by using the  Mehler's formula for $\Phi_N^{(d+1)}(x,y)$, $x,y\in \R^{d+1}$, we have
\begin{multline*}
\sum_{N=0}^\infty w^N\int_{-1}^1\Phi_N^{(d+1)}(r,s;u)Q_m^{\la(1/2+i\beta)}(u)\,du\\
 = \frac{\Gm(d/2)}{4\pi^{d/2}} (1-w^2)^{i\beta}\sum_{N=0}^{\infty}(rs)^m  \psi_{N}^{\la(1/2+i\beta)+m}(r)\psi_{N}^{\la(1/2+i\beta)+m}(s)w^{2N+m}.
 \end{multline*}
Proceeding as above with this identity we obtain \eqref{nuc2}.

This completes the proof of the proposition.
\end{proof}
\section{Technical results}
\label{sec:technical}

\subsection{A lemma concerning binomial coefficients with complex parameters}
\label{subsec:binomial-complex}

The estimate contained in the following lemma is used in the proof of Theorem \ref{maximal}. Moreover, a slightly different version of such lemma (whose explicit statement and proof are omitted) is also used in the proof of Proposition \ref{prop:estimate}.

\begin{lem}
\label{lem:bin.est} For $\zeta\in \C$, let $\dl(\zeta)$ be defined as in \eqref{zetaC}. For any $ \ve > 0 $ and $ \beta \in \R $ we have the estimate
$$ \frac{1}{\big |A_N^{\dl(i\bt+\varepsilon)} \big |} \sum_{j=0}^N\big |A_{N-j}^{i\bt+\frac{\varepsilon}{2}-1} \big | A_j^{\dl(\varepsilon/2)} \leq C_{\varepsilon} (1+|\bt|) \cosh (\pi \beta) $$
where $ C_\ve $ is independent of $ N.$
\end{lem}
\begin{proof}
 In order to prove the lemma we make use of the fact that
 \begin{equation} \label{gammaFunc}
 |\Gm(\ap+i\bt)|\leq \Gm(\ap),
 \end{equation}
 for $\ap,\, \bt \in \R,\; \ap>0 $, which follows from the very definition of Gamma function. We also have the Beta function
 $$\frac{\Gm(x)\Gm(y)}{\Gm(x+y)}=\int_0^1(1-t)^{x-1}t^{y-1}\,dt $$
 which leads to the estimate
 \begin{equation}
 \label{beta}\Big | \frac{\Gm(\ap)\Gm(i\bt+1)}{\Gm(\ap+i\bt+1)}\Big |\leq \frac{\Gm(\ap)}{\Gm(\ap+1)}
 \end{equation}
 for any $\ap>0$. We rewrite
 $$\sum_{j=0}^{N}A_{N-j}^{i\bt+\ve/2-1}A_j^{\dl(\ve/2)} = \frac{\Gm(\ve/2)}{\Gm(\ve/2+i\bt)} \sum_{j=0}^N \frac{\Gm(N-j+\ve/2+i\bt)}{\Gm(N-j+1)\Gm(\ve/2)}A_j^{\dl(\ve/2)}. $$
 By \eqref{gammaFunc} and \eqref{binomIdent} we have
 \begin{align*}
 \Big |\sum_{j=0}^{N}A_{N-j}^{i\bt+\ve/2-1}A_j^{\dl(\ve/2)} \Big |&\leq \frac{\Gm(\ve/2)}{|\Gm(\ve/2+i\bt)|}\sum_{j=0}^N
 \frac{\Gm(N-j+\ve/2)}{\Gm(N-j+1)\Gm(\ve/2)}A_j^{\dl(\ve/2)}\\
 &=\frac{\Gm(\ve/2)A_N^{\dl(\ve)}}{|\Gm(\ve/2+i\bt)|}.
 \end{align*}
 Thus, we are left with estimating
$$\bigg|\frac{\Gm(\ve/2)}{\Gm(\ve/2+i\bt)}\;\frac{A_N^{\dl(\ve)}}{A_N^{\dl(\ve+i\bt)}}\bigg| =\bigg|\frac{\Gm(\ve/2)}{\Gm(\ve/2+i\bt)}\; \frac{\Gm(\dl(\ve)+i\bt+1)}{\Gm(\dl(\ve)+1)}\; \frac{\Gm(N+\dl(\ve)+1)}{\Gm(N+\dl(\ve)+i\bt+1)}\bigg|. $$
The middle term in the right hand side is clearly bounded by \eqref{gammaFunc}. The first term can be written as
$$\frac{(\ve/2+i\bt)\Gm(\ve/2)\Gm(i\bt+1)}{\Gm(i\bt+1)\Gm(i\bt+\ve/2+1)} $$
which leads, by \eqref{beta}, to the estimate
$$ \bigg|\frac{\Gm(\ve/2)}{\Gm(\ve/2+i\bt)}\bigg|\le\frac{\Gm(\ve/2)}{\Gm(\ve/2+1)}\frac{(1+|\bt|)} { |\Gm(i\bt+1)|}. $$
Similarly, the third term gives the estimate $|\Gm(i\bt+1)|^{-1} $. Therefore, the expression in Lemma \ref{lem:bin.est} is bounded by
$$\frac{\Gm(\ve/2)}{\Gm(\ve/2+1)}(1+|\bt|)\;|\Gm(i\bt+1)|^{-2}. $$
Once again the expression with the Beta function
$$\frac{\Gm\big(\frac12\big)\Gm\big(\frac12+i\bt\big)}{\Gm(1+i\bt)} =\int_0^1(1-t)^{-\frac12}t^{i\bt-\frac12}\,dt$$
leads to the estimate $|\Gm(i\bt+1)|^{-1}\leq \sqrt{\pi}|\Gm(\tfrac12+i\bt)|^{-1} $. Therefore,
$$|\Gm(i\bt+1)|^{-2}\leq\pi\,\big|\Gm\big(\tfrac12+i\bt\big)\big|^{-2}=\pi\, \big(\Gm(\tfrac12+i\bt)\Gm(\tfrac12-i\bt)\big)^{-1}.$$
In view of the identity $\displaystyle \Gm(z)\Gm(1-z)=\tfrac{\pi}{\sin{\pi z}}$ we obtain
$$|\Gm(i\bt+1)|^{-2}\leq \sin{\pi \big(\tfrac12+i\bt \big)}=\cosh{\pi \bt}. $$
This completes the proof of the lemma.
\end{proof}

\subsection{Some results on ultraspherical polynomials and Bessel functions}
\label{sec:ultras}

We will prove some technical results involving ultraspherical polynomials and modified Bessel
functions with complex parameters. Let us recall several facts concerning both special functions.

Rodrigues' formula for the normalized ultraspherical polynomials takes the form
\begin{equation}\label{eq:Rodrigue}
(1-u^2)^{\la-1/2}P_m^{\la}(u)=\frac{(-1)^m}{2^m (\la+1/2 )_m}\,\frac{d^m}{du^m} \big ((1-u^2)^{m+\la-1/2}\big ), \qquad m=0,1,\ldots,
\end{equation}
where $(\la)_m=\tfrac{\Gamma(\la+m)}{\Gamma(\la)}$ is the Pochhammer symbol. The normalized ultraspherical polynomials are also given by the explicit formula $ P_m^{\la}(\cos\theta)= \frac{C_m^{\la}(\cos\theta)}{C_m^{\la}(1)} $ where (see \cite[page 302]{AAR})
\begin{equation}\label{eq:UltraExp}
C_m^{\la}(\cos\theta)=\sum_{k=0}^{m}\frac{(\la)_k(\la)_{m-k}}{k!(m-k)!}\cos{(m-2k)\theta}
\end{equation}
for $\theta \in [0,\,\pi] $. It is immediate from \eqref{eq:UltraExp} that
\begin{equation}\label{acot}
|P_m^{\la}(u)|\leq 1.
\end{equation}
Since the functions $\la\mapsto (\la)_k $ are holomorphic  for each fixed $u \in [-1,\,1]$, the ultraspherical polynomials $P_m^{\la}(u)$ can be extended as a holomorphic function of $\la$ on the domain $\{\la \in \C: \Re{(\la)}>-\tfrac{1}{2} \}$.

Let $I_{\la}$ be the modified Bessel function of the first kind and order $\la$, defined as
$$  I_\la(z) = \sum_{m=0}^{\infty} \frac{1}{m!\,\Gamma(m+\la+1)} \bigg(\frac{z}{2}\bigg)^{2m+\la}.$$
Note that the functions $I_{\lambda} $ can also be defined for complex values of $\lambda.$ We will use Schl\"afli's  integral representation of Poisson type for modified Bessel function, see \cite[(5.10.22)]{Lebedev}
\begin{equation}\label{eq:ModBessel}
I_{\la}(z)=\frac{(z/2)^{\la}}{\sqrt{\pi}\Gm(\la+1/2)}\int_{-1}^1e^{zu}(1-u^2)^{\la-1/2}du, \quad |\arg z|<\pi, \,\,\, \la>-\frac12.
\end{equation}

The following lemma is useful to express the Dunkl--Ces\`aro mean kernel in terms of the Ces\`{a}ro kernel for the Laguerre expansions in Subsection \ref{LaguerreCon}.

\begin{lem}\label{lem:UltraBessel}
Let $z\in \C$ and $\la>-\frac{1}{2}$. Then the following holds
\begin{equation}\label{eq:UltraBessel}
 \int_{-1}^{1}e^{zu}P_m^{\la}(u)(1-u^2)^{\la-1/2}\,du =\sqrt{\pi}\Gm (\la+1/2 )(z/2 )^{-\la}I_{\la+m}(z), \quad  m=0,1,\ldots.
\end{equation}
\end{lem}
\begin{proof}
 In view of Rodrigues' formula \eqref{eq:Rodrigue}, we get $$\int_{-1}^{1}e^{zu}P_m^{\la}(u)(1-u^2)^{\la-1/2}\,du
=  \frac{(-1)^m}{2^m(\la+1/2)_m}\,\int_{-1}^{1}e^{zu}\frac{d^m}{du^m} ((1-u^2)^{m+\la-1/2} )\,du$$
Integrating by parts and using \eqref{eq:ModBessel} we see that
\begin{align*}
\frac{(-1)^m}{2^m (\la+1/2)_m}\,\int_{-1}^{1}e^{zu}\frac{d^m}{du^m} ((1-u^2)^{m+\la-1/2} )\,du &= \frac{(z/2)^m}{
(\la+1/2)_m}\,\int_{-1}^{1}e^{zu} (1-u^2)^{m+\la-1/2}\,du\\
&=\sqrt{\pi}\Gm (\la+1/2 )(z/2 )^{-\la}I_{\la+m}(z)
\end{align*}
since $\Gm(\la+1/2)=\frac{\Gm(\la+m+1/2)}{(\la+1/2)_m}$. This proves the lemma.
\end{proof}

In order to estimate the Ces\`{a}ro kernels for Dunkl--Hermite expansions, we needed to consider a variant of Lemma \ref{lem:UltraBessel} where the parameter $ \lambda $ has to be taken complex. In the following lemma we express $I_{\lambda+m}$ in terms of certain variants of ultraspherical polynomials defined for complex $\lambda$. Furthermore, we also obtain some good estimates for these variants.

For $u\in[-1,1]$, define
\begin{equation*} \label{eq:Q}
Q_m^{\la,\ve}(u):= \frac{1}{\Gamma(i\bt+\ve)} \int_0^1\chi_{[|u|,\,1]}(s)P_m^{\ap}\Big(\frac{u}{s}\Big) (s^2-u^2)^{\ap-1/2}(1-s)^{i\bt+\varepsilon-1}s^{m+1}\,ds
\end{equation*}
for $ m =0,1,2,\dots$.

\begin{lem}\label{lem:Q}
Let $\varepsilon>0$ and $\la=\ap+i\bt $ with $\ap>\frac{1}{2}$ and $\bt \in \R$.  Then, for $z\in \C$, we have the identity
\begin{equation}\label{eq:Q1}  \int_{-1}^{1}e^{zu}Q_m^{\la,\ve}(u)\,du=\frac{\sqrt{\pi}\Gm(\ap+1/2)}{2}\cdot
\frac{I_{\la+m+\ve}(z)}{(z/2)^{\la+\ve}}.
\end{equation}
Moreover, we have the uniform estimates
\begin{equation*}\label{estimComplex}
|Q_m^{\la,\ve}(u)|\leq C\frac{(1-u^2)^{\ap-\frac{1}{2}}}{|\Gamma(i\bt+1)|},
\end{equation*}
for all $-1\leq u \leq 1$.
Consequently,
\begin{equation*}
Q_m^\la(u) := \lim_{\ve \rightarrow 0}Q_m^{\la,\ve}(u)
\end{equation*}
exists, satisfies \eqref{estimComplex} and
\begin{equation}\label{eq:BesselUltra}
 \int_{-1}^{1}e^{zu}Q_m^{\la}(u)\,du=\frac{\sqrt{\pi}\Gm(\ap+1/2)}{2}\cdot \frac{I_{\la+m}(z)}{(z/2)^{\la}}.
\end{equation}
\end{lem}
\begin{proof}
 For $\Re(\nu)>0$ and $\Re(\mu)>0$, we have (see \cite[Theorem 4.11.1]{AAR})
$$ \frac{J_{\nu+\mu}(z)}{z^{\nu+\mu}}=\frac{2^{1-\mu}}{\Gm(\mu)}\int_0^1 \frac{J_{\nu}(sz)}{(sz)^{\nu}}(1-s^2)^{\mu-1}s^{2\nu+1}\,ds,$$
where $J_{\nu}$ is the Bessel function of order $\nu$. Using the relation $\frac{J_{\nu}(iz)}{(iz/2)^{\nu}}=\frac{I_{\nu}(z)}{(z/2)^{\nu}}$ and taking $\nu=\ap+m $ and $\mu=\varepsilon+i\bt,$ i.e. $\nu+\mu=\la+m+\ve$, the above can be written as \begin{align*}
\frac{I_{\la+m+\ve}(z)}{(z/2)^{\la+m+\ve}}&=\frac{2}{\Gm(i\bt+\varepsilon)}
\int_0^1 \frac{I_{\ap+m}(sz)}{(sz/2)^{\ap+m}}(1-s^2)^{i\bt+\varepsilon-1}s^{2\ap+2m+1}\,ds\\
&=\frac{2}{\Gm(i\bt+\varepsilon)}\int_0^1 \frac{I_{\ap+m}(sz)}{(sz/2)^{\ap}}(1-s^2)^{i\bt+\varepsilon-1}s^{2\ap+m+1}\,ds.
\end{align*}
Therefore, in view of this and \eqref{eq:UltraBessel}, we get
$$ \frac{I_{\la+m+\ve}(z)}{(z/2)^{\la+\ve}}=\frac{2}{\Gm(i\bt+\varepsilon)\sqrt{\pi}\Gm(\ap+1/2)}
\int_0^1 \int_{-1}^1e^{szu}P_m^{\ap}(u)(1-u^2)^{\ap-\frac{1}{2}}\,du\,(1-s^2)^{i\bt+\varepsilon-1}s^{2\ap+m+1}\,ds.$$
By making a change of variables and then a change of the order of integration, the right hand side of the above can be written as
$$\frac{2}{\Gm(i\bt+\varepsilon)\sqrt{\pi}\Gm(\ap+1/2)} \int_{-1}^1 e^{zu}\int_0^1 \chi_{[|u|,\,1]}(s)P_m^{\ap}\Big(\frac{u}{s}\Big)(s^2-u^2)^{\ap-\frac{1}{2}}(1-s^2)^{i\bt+\varepsilon-1}s^{m+1}\,ds\,du.$$
Thus we have the desired identity
$$\frac{I_{\la+m+\ve}(z)}{(z/2)^{\la+\ve}}=\frac{2}{\sqrt{\pi}\Gm(\ap+1/2)} \int_{-1}^1 e^{zu}Q_m^{\la,\ve}(u)\,du.$$

We now proceed to show that $\lim_{\ve \rightarrow 0} Q_m^{\la,\ve}(u)$ exists and satisfies the required estimate. In
 order to do so, let us write
$$f_m(s)=P_m^{\ap}(s)(1-s^2)^{\ap-\frac12},\quad g_m(s)=s^{2\ap+m-1} $$
so that we can write
$$ Q_m^{\la,\ve}(u)=\frac{i\bt+\ve}{\Gm(i\bt+\ve+1)}\int_{|u|}^1f_m\big(\tfrac{u}{s}\big)(1-s^2)^{i\bt+\ve-1}g_m(s)s\,ds. $$ Integrating by parts and noting that the boundary terms vanish (since $\ap > \tfrac12$) we have
$$ Q_m^{\la,\ve}(u)=\frac{1}{2\Gm(i\bt+\ve+1)}\int_{|u|}^1\frac{d}{ds}\Big (f_m\Big(\frac{u}{s}\Big)g_m(s)\Big )(1-s^2)^{i\bt+\ve}\,ds. $$
Observe that $\displaystyle| Q_m^{\la,\ve}(u)|\le \tfrac{1}{2|\Gm(i\bt+1)|}\int_{|u|}^1\big|\tfrac{d}{ds} \big (f_m\big(\tfrac{u}{s}\big)g_m(s)\big )\big|(1-s^2)^{i\bt}\,ds.$  Besides, it is easy to check that we can now pass to the limit as $\ve \rightarrow 0$ and define
$$ Q_m^{\la}(u):=\lim_{\ve \rightarrow 0}Q_m^{\la,\ve}(u)=\frac{1}{2\Gm(i\bt+1)}\int_{|u|}^1\frac{d}{ds}\Big (f_m\Big(\frac{u}{s}\Big)g_m(s)\Big )(1-s^2)^{i\bt}\,ds. $$
On one hand, the boundedness \eqref{acot} of $P_m^{\ap}(u)$ leads to the estimate
$$\int_{|u|}^1\Big|f_m\Big(\frac{u}{s}\Big)\Big|g'_m(s)\,ds\leq C \big(g_m(1)-g_m(|u|)\big)\leq 2C.$$
On the other hand, from Rodrigues' formula \eqref{eq:Rodrigue} it is easy to see that $|f'_m(u)|\leq C(1-u^2)^{\ap-\frac32} $  and therefore,
$$\int_{|u|}^1\Big|f'_m\Big(\frac{u}{s}\Big)\Big|\frac{|u|}{s^2}\;g_m(s)\,ds \leq C \int_{|u|}^1 (s^2-u^2)^{\ap-\frac32} \frac{|u|}{s}s^{m+1} ds$$
which gives the estimate (as  $|u| \leq s\leq 1$)
$$|Q_m^{\la,\ve}(u)|\leq \frac{C}{|\Gm(i\bt+1)|}(1-u^2)^{\ap-\tfrac12}. $$
As $ Q_m^{\la}(u) $ is defined as the limit of $ Q_m^{\la,\ve}(u) $ it follows that $ Q_m^{\la}(u)$ also satisfies the same estimate.
Passing to the limit in \eqref{eq:Q1} we see that
$$\int_{-1}^{1}e^{zu}Q_m^{\la}(u)\,du=\frac{\sqrt{\pi}\Gm(\ap+1/2)}{2}\cdot \frac{I_{\la+m}(z)}{(z/2)^{\la}}. $$
This completes the proof.
\end{proof}

\subsubsection*{Acknowledgments.}
This work was mainly developed when the second author was visiting the Indian Institute of Science, Bangalore.
She is grateful to the Department of Mathematics for the warm hospitality. The authors are immensely grateful to the referee for his thorough reading of the manuscript and useful suggestions which have been incorporated in the revised version.


\end{document}